\documentclass[10pt,a4paper,reqno]{amsart} 

\usepackage{latexsym}
\usepackage{verbatim}
\usepackage{epsfig}
\usepackage{rotating}
\usepackage{amssymb}
\usepackage[T1]{fontenc}
\usepackage{afterpage}
\usepackage{color}
\usepackage{dsfont}

\usepackage{amssymb,amsfonts,amsmath,stmaryrd}

\catcode`\@=11
\def\section{\@startsection{section}{1}%
 \z@{.7\linespacing\@plus\linespacing}{.5\linespacing}%
 {\normalfont\bfseries\scshape\centering}}

\def\subsection{\@startsection{subsection}{2}%
  \z@{.5\linespacing\@plus\linespacing}{.5\linespacing}%
  {\normalfont\bfseries\scshape}}

\def\subsubsection{\@startsection{subsubsection}{3}%
 \z@{.5\linespacing\@plus\linespacing}{-.5em}
  {\normalfont\bfseries\itshape}}
\catcode`\@=12

\usepackage{amssymb,amsfonts,amsmath,stmaryrd,color,epsfig}

\usepackage{epsfig,color}
\usepackage{rotating}
\usepackage{amssymb}
\usepackage[T1]{fontenc}
\usepackage{afterpage}

\newcommand{\zs}{\mathbb{Z}}

\newcommand{\rs}{\mathbb{R}}

\def\emm#1,{{\em #1}}

\newcommand{\Ref}[1]{(\ref{#1})}
\newcommand{\beq}{\begin{equation}}
\newcommand{\eeq}{\end{equation}}

\newtheorem{Theorem}{Theorem}
\newtheorem{propo}[Theorem]{Proposition}
\newtheorem{coro}[Theorem]{Corollary}
\newtheorem{Lemma}[Theorem]{Lemma}
\newtheorem{Definition}[Theorem]{Definition}

\graphicspath{{Figures/}}

\title[Perfect matchings for the three-term Gale-Robinson sequences]
{Perfect matchings\\
 for the three-term Gale-Robinson sequences }

\author{Mireille Bousquet-M\'elou}
\address{CNRS, LaBRI, Universit\'e Bordeaux 1, 351 cours de la Lib\'eration,
  33405 Talence Cedex, France}
\email{bousquet@labri.fr}

\author{James \sc{Propp}}
\address{University of Massachusetts Lowell, MA 01854, USA}
\email{\small\tt JamesPropp@gmail.com}
\thanks{JP was supported by grants from the National Security Agency and the National Science Foundation.
JW was supported by the National Sciences and Engineering Research Council of Canada.}

\author{Julian \sc{West}}
\address{University of Victoria, PO Box 3060,
Victoria, BC V8W3R4, Canada}
\email{\small\tt julian@julianwest.ca}

\date{16 June 2009}

\begin{document}


\begin{abstract}
In 1991, David Gale and Raphael Robinson, 
building on explorations carried out by Michael Somos in the 1980s,
introduced a three-parameter family of rational recurrence relations,
each of which (with suitable initial conditions) appeared
to give rise to a sequence of integers, even though a priori
the recurrence might produce non-integral rational numbers.
Throughout the '90s, proofs of integrality were known only
for individual special cases.
In the early '00s, Sergey Fomin and Andrei Zelevinsky
proved Gale and Robinson's integrality conjecture. They actually
proved much more, and in particular, that certain bivariate rational
functions that generalize 
Gale-Robinson numbers are actually polynomials with integer
coefficients. 
However, their proof
did not offer any enumerative interpretation of the Gale-Robinson
numbers/polynomials.
Here we provide such an interpretation in the
setting of perfect matchings of graphs, which makes
integrality/polynomiality obvious. Moreover,
this interpretation implies
that the coefficients of the Gale-Robinson polynomials
are positive, as Fomin and Zelevinsky conjectured.
\end{abstract}
\maketitle

\begin{center}
{\it In memory of David Gale, 1921-2008}
\end{center}

\section{Introduction}
%
Linear recurrences 
are ubiquitous in combinatorics, as part of a broad general framework 
that is well-studied and well-understood; 
in particular, many combinatorially-defined sequences
can be seen on general principles to satisfy linear recurrences
(see~\cite{Stanley}),
and conversely, when an integer sequence 
is known to satisfy a linear recurrence
it is often possible to reverse-engineer
a combinatorial interpretation for the sequence
(see~\cite{berstel-reutenauer} and references therein for a general
discussion, and~\cite[Chapter~3]{BQ} for specific examples).
In contrast, rational recurrences such as 
$$
s(n) = ( s(n-1)s(n-3) + s(n-2)^2)/s(n-4),
$$
which we prefer to write in the form
$$
s(n)s(n-4) = s(n-1)s(n-3) + s(n-2)^2,
$$
are encountered far less often, 
and there is no simple general theory 
that describes the solutions to such recurrences
or relates those solutions to combinatorial structures.
The particular rational recurrence relation given above
is the Somos-4 recurrence, 
and is part of a general family of recurrences introduced by Michael Somos:
$$
s(n) s(n-k) = s(n-1)s(n-k+1) + s(n-2)s(n-k+2) + \cdots + 
s(n -\lfloor k/2 \rfloor) s(n - \lceil k/2 \rceil).
$$
If one puts $s(0)=s(1)=\cdots=s(k-1)=1$
and defines subsequent terms using the Somos-$k$ recurrence,
then one gets a sequence of rational numbers
which for the values $k=4,5,6,7$
is actually a sequence of integers.
(Sequences Somos-4 through Somos-7 are entries
A006720 through A006723 in~\cite{sloane}.)
Although integer sequences satisfying such recurrences have received 
a fair bit of attention in the past few years, 
until recently algebra remained one step ahead of combinatorics, 
and there was no enumerative interpretation of these integer sequences. 
(For links related to Somos sequences, see
{\tt http://jamespropp.org/somos.html}.)

Inspired by the work of Somos, 
David Gale and Raphael Robinson~\cite{Gale,Galeupdate}
considered sequences given by recurrences of the form
$$a(n) a(n-m) = a(n-i)a(n-j)+a(n-k)a(n-\ell),$$ 
with initial conditions $a(0) = a(1) = \cdots = a(m-1) = 1$,
where $m=i+j=k+\ell$.
We call this the \emm three-term Gale-Robinson recurrence,
\footnote{Gale and Robinson also considered recurrences of the form
$a(n) a(n-m) = a(n-g)a(n-h)+a(n-i)a(n-j)+a(n-k)a(n-\ell)$ 
for suitable values of $g,h,i,j,k,\ell,m$,
but such \emm four-term Gale-Robinson recurrences,
will not be our main concern here.}.
The Somos-4 and Somos-5 recurrences 
are the special cases where $(i,j,k,\ell)$ is equal to 
$(3,1,2,2)$ and $(4,1,3,2)$ respectively.
Gale and Robinson conjectured
that for all integers $i,j,k,\ell > 0$
with $i+j=k+\ell=m$, the sequence $a(0),a(1),\dots$
determined by this recurrence
has all its terms given by integers.
About ten years later, this was 
proved algebraically in an influential paper by Fomin and 
Zelevinsky~\cite{FZ}. 

\subsection{Contents}
In this paper, we 
first give a \emm combinatorial, proof of the integrality of
the three-term Gale-Robinson sequences.
The integrality comes as a side-effect of producing a
combinatorial interpretation of those sequences.
Specifically, we construct a sequence of graphs $P(n;i,j,k,\ell)$
($n \geq 0$) and prove in Theorem~\ref{thm:GR-pinecones} that
the $n$th graph in the sequence has $a(n)$ (perfect) matchings.  
Our graphs, which we call {\it pinecones}, 
generalize the well-known Aztec diamond graphs, 
which are the matchings graphs for the Gale-Robinson sequence 
1, 1, 2, 8, 64, 1024, \dots \, in which $i=j=k=\ell=1$.  
A more generic example of a pinecone is shown in
Figure~\ref{big-example1}. 
All pinecones are subgraphs of the square grid.

\begin{figure}[htb]
\begin{center}
\scalebox{0.9}{\input{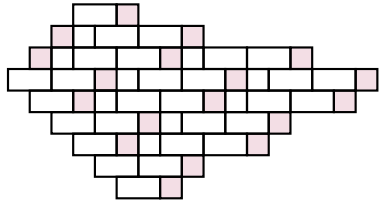}}
\caption{The pinecone $P(25;6,2,5,3 )$. Its matching number is
$a(25)$, where $a(n)$ is the Gale-Robinson sequence associated with
$(i,j,k,\ell)=(6,2,5,3)$. }
\label{big-example1}
\end{center}
\end{figure}

We give two ways to construct pinecones
for the Gale-Robinson sequences:
a recursive method (see Figure~\ref{fig:somos4-rec} and the surrounding text)
that constructs the graph $P(n;i,j,k,\ell)$
in terms of the smaller graphs $P(n';i,j,k,\ell)$ with $n' < n$,
and a direct method (see Formula~\eqref{UL-def}
in Section~\ref{sec:pinecones-GR})
that allows one to construct the graph $P(n;i,j,k,\ell)$ immediately.
The heart of our proof is the demonstration that if one defines $a(n)$
as the number of perfect matchings of $P(n)\equiv P(n;i,j,k,\ell)$,
the sequence $a(0),a(1),a(2),...$ satisfies the Gale-Robinson recurrence.
This fact, in combination with a simple check
that $a(0)=a(1)=\cdots=a(m-1)=1$,
gives an immediate inductive validation
of our claim that $P(n)$ has $a(n)$ perfect matchings for all $n$,
which yields additionally the integrality of $a(n)$.

General pinecones are defined in Section~\ref{sec:pinecones}, where we also
explain how to compute inductively their matching number via Kuo's
condensation lemma~\cite{Kuo}.
In Section~\ref{sec:pinecones-GR}, we describe how
to associate a sequence of pinecones to a Gale-Robinson sequence, and
observe that for these pinecones, the condensation lemma specializes
precisely to the Gale-Robinson recurrence.
Indeed, the recursive method of constructing pinecones,
in combination with Kuo's condensation lemma,
gives combinatorial meaning to the
different terms $a(n_1) a(n_2)$ of the Gale-Robinson recurrence.

In Section~\ref{sec:bivariate}, we refine our argument to prove that
the sequence $p(n)\equiv p(n;w,z)$ defined by
$$
p(n)p(n-m)=w\,p(n-i)p(n-j)+z\,p(n-k)p(n-\ell),
$$
with $i+j=k+\ell=m$ and $p(0)=p(1)=\cdots=p(m-1)=1$, is a sequence of
polynomials in $w$ and $z$ with nonnegative integer coefficients.
More precisely, we prove in Theorem~\ref{thm:GR-pinecones-refined}
that $p(n;u^2,v^2)$ counts perfect matchings of
the pinecone $P(n;i,j,k,\ell)$ by the number of \emm special,
horizontal edges (the exponent of the variable $u$) and the number 
of vertical edges (the exponent of the variable $v$). 
The fact that $p(n)$ is a polynomial with coefficients
in $\zs$ was proved in~\cite{FZ}, but no combinatorial explanation was
given and the non-negativity of the coefficients was left open.

\subsection{Strategy, and connections with previous work}
For much of the work in this paper, we share precedence with 
the students in the NSF-funded program REACH (Research Experiences 
in Algebraic Combinatorics at Harvard), led by James Propp, whose
permanent archive is on the web at {\tt http://jamespropp.org/reach/}.
A paper by one of these students, David Speyer~\cite{Sp}, 
introduced a very flexible framework (the ``crosses and wrenches method'') 
that, starting from a recurrence relation of a certain type,
constructs a sequence of graphs 
whose matching numbers satisfy the given recurrence.
This framework includes the three-term Gale-Robinson recurrences, and
thus yields a combinatorial proof of the integrality of the associated
sequences. This extends to a proof that the bivariate Gale-Robinson polynomials
mentioned above are indeed polynomials, and have non-negative coefficients.
One difference with our paper is that Speyer's graphs are only
described explicitly for Somos-4 and Somos-5 sequences, whereas our
construction is explicit for any Gale-Robinson sequence.
Moreover, the description of
our graphs as subgraphs of the square grid looks more regular, and may
be useful to study limit shapes of random perfect matchings
such as the perfect matchings shown in Figure~\ref{fig:random}.

Let us mention that shortly after Speyer did his work on perfect matchings,
he and his fellow REACH-participant Gabriel Carroll 
did for four-term Gale-Robinson recurrences 
what Speyer had done for three-term Gale-Robinson recurrences,
by introducing new objects called ``groves'' 
to take the place of perfect matchings~\cite{CS}.
Carroll and Speyer's work gives, as two special cases, 
combinatorial proofs of the integrality of Somos-6 and Somos-7.

\medskip

The strategies that led to 
Speyer's article~\cite{Sp} and to the present article
are not entirely independent;
each made use of Propp's prior construction
of a suitable perturbed Gale-Robinson recurrence,
which we explain next. 
The explanation will mostly be of interest
to researchers seeking to apply similar techniques
to other problems; others may want to skip the rest of the introduction.

Suppose we perturb a three-term Gale-Robinson recurrence
by replacing the singly-indexed Gale-Robinson number $a(n)$ 
by a triply-indexed quantity $A(n,p,q)$ 
satisfying the perturbed recurrence
$$
A(n,p,q)A(n-m,p,q)=A(n-i,p-1,q)A(n-j,p+1,q)+A(n-k,p,q+1)A(n-\ell,p,q-1).
$$
(This choice of perturbation is not as special as it looks:
all that matters is that the pairs $(-1,0),(1,0),(0,1),(0,-1)$
that describe the perturbations of the second and third coordinates
in the four index-triples on the right-hand side,
viewed as points in the plane,
form a non-degenerate centrally-symmetric parallelogram.
Choosing a different centrally-symmetric parallelogram
is tantamount to a simple re-indexing of the recurrence.)
If we take as our initial conditions $A(n,p,q) = x_{n,p,q}$
for all $n$ between 0 and $m-1$ and $p,q$ arbitrary,
with (formal) indeterminates $x_{n,p,q}$,
then each $A(n,p,q)$ with $n \geq m$ can be expressed 
as a rational function of these indeterminates.
It should be emphasized here
that for all $n,p,q,r,s$,
the rational functions $A(n,p,q)$ and $A(n,r,s)$
are the same function up to re-indexing of the indeterminates.

Propp conjectured that each $A(n,p,q)$ is a Laurent polynomial 
in some finite subset of the (infinitely many) indeterminates
$x_{n,r,s}$, with integer coefficients;
that is, each $A(n,p,q)$ is an element of $\zs[x_{n,r,s}^{\pm 1}]$.
This was subsequently proved by Fomin and Zelevinsky~\cite{FZ}.
Note that if one sets all the indeterminates $x_{n,r,s}$ equal to 1,
the Laurent polynomials $A(n,p,q)$ specialize to the Gale-Robinson
numbers $a(n)$.  Propp conjectured that 
each coefficient in each such Laurent polynomial is positive
(a fact that is not proved by Fomin and Zelevinsky's method)
and furthermore is equal to $1$.

Propp knew
that in the case $i=j=k=\ell=1$,
the Laurent polynomials $A(n,p,q)$ can be interpreted as 
multivariate matching polynomials of suitable graphs,
namely, the Aztec diamond graphs.
(See Subsection~\ref{sec:preliminaries} for a definition of
matching polynomials.) 
Indeed, David Robbins had studied
the three-parameter ``perturbed recurrence'' in this case,
on account of its relation to the study of determinants,
and had shown (with Rumsey)~\cite{RR}
that the associated rational functions are Laurent polynomials.
(For more background on this connection with determinants, see~\cite{BP}.)
The work by Elkies, Kuperberg, Larsen, and Propp~\cite{EKLP}
had shown that the monomials in these Laurent polynomials
correspond to perfect matchings of Aztec diamond graphs.
So it was natural to hope that this correspondence
could be extended to the Gale-Robinson family of recurrences.

It should be acknowledged here that the idea 
behind the specific triply-indexed perturbation $A(n,p,q)$
of the Gale-Robinson sequence that proved so fruitful
came from an article of Zabrodin~\cite{Z}
that was brought to Propp's attention by Rick Kenyon.
This article led Propp to think that the recurrence studied by Robbins
should be considered a special case of
the ``discrete bilinear Hirota equation'', or ``octahedron equation'',
and that other recurrences such as the Gale-Robinson recurrence
should likewise be considered in the context of the octahedron equation.

What the REACH students were able to do,
after diligent examination of the Laurent polynomials $A(n,p,q)$,
is view those Laurent polynomials as
multivariate matching polynomials of suitable graphs.
Bousquet-M\'elou and West, independently, did the same for small
values of $n$, until they were able to extrapolate these examples to
the generic form of the graphs, which became the pinecones of this paper.

There is a general strategy here for
reverse-engineering combinatorial interpretations of
algebraically-defined sequences of numbers:
add sufficiently many extra variables so that the numbers 
become Laurent polynomials in which every coefficient equals 1.
For another application of this reverse-engineering method
(in the context of Markoff numbers and frieze patterns),
see~\cite{Pr2}.

\section{Perfect matchings of pinecones} \label{sec:pinecones}

In this section we define a family of subgraphs of the square
lattice, which we call \emph{pinecones}.  Then we prove that the
number of perfect matchings of a pinecone can be computed inductively
in terms of the number of perfect matchings of five of its
sub-pinecones.

\subsection{Preliminaries}  \label{sec:preliminaries}
To begin with, let us recall some terminology about graphs. A (simple)
graph $G$ is an ordered pair $(V,E)$ where $V$ is a finite set of
\emm vertices,, and $E$, the set of \emm edges,, is 
a collection of 2-element subsets of $V$.
The \emm degree, of a vertex $v$ is the number of edges in $E$
containing $v$. 
A \emm subgraph, of $G$ is a graph $H=(V', E')$ such that
$V' \subset V$ and $E' \subset E$. 
If, in addition, $V'=V$, we say that $H$ is a 
\emm spanning subgraph, of $G$. 
The \emm intersection,  of two graphs $G=(V,E)$ and $H=(V',E')$ 
is the graph $G\cap H= (V \cap V' , E\cap E')$,
and the \emm union,  of the two graphs 
is the graph $G\cup H= (V \cup V' , E\cup E')$.
Given two graphs $G=(V,E)$ and $H=(V',E')$, 
we denote by $G\setminus H$ the subgraph
$(V'', E'')$, where $V''=V\setminus V'$ and $E''$ is the set of edges
of $E\setminus E'$ having both endpoints in $V''$.

A \emm perfect matching, of a graph $G=(V,E)$ is a subset $E'$ of $E$ such
that every vertex of $V$ belongs to exactly one edge of $E'$. We will
sometimes omit the word ``perfect'' and refer to perfect 
matchings as simply ``matchings''. The
\emm matching number, of $G$, denoted by $m(G)$, is the number of perfect
matchings of $G$. 
More generally, we shall often consider the set $E$ of edges as a set of
commuting indeterminates, and associate with a (perfect) matching $E'$ the
product of the edges it contains.  The \emm matching polynomial, of $G$ is
thus defined to be
$$
M(G):= \sum_{E'} \prod_{e \in E'} e,
$$
where the sum runs over all perfect matchings $E'$ of $G$.
If we replace every $e$ that occurs in this sum-of-products by 
a non-negative integer $n_e$, then this expression becomes
a non-negative integer, namely, the number of perfect matchings
of the multigraph in which there are $n_e$ edges joining the
vertices $x$ and $y$ for all $e = \{x,y\}$ in $E$ (and 
no edges joining $x$ and $y$ if $\{x,y\}$ is not in $E$).
In particular, if each $n_e$ is set equal to 0 or 1, 
then the matching polynomial becomes the number of perfect
matchings of the subgraph of $G$ consisting of precisely 
those edges $e$ for which $n_e = 1$.

\begin{figure}[thb]
\begin{center}
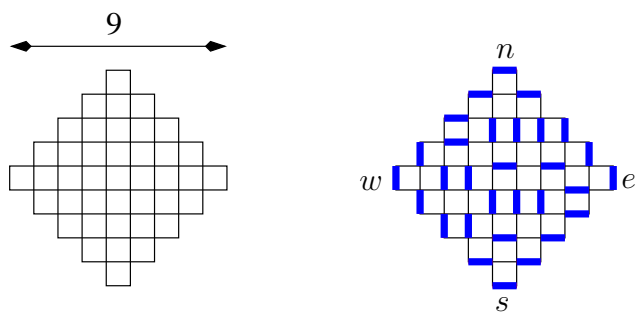
\end{center}
\caption{An Aztec diamond graph
of width $9$, and one of its perfect matchings.}
\label{fig:aztec}
\end{figure}

\subsection{Aztec diamonds graphs}
The {pinecones} considered in this paper are certain subgraphs of
the square lattice. The most regular of them are the
\emph{(Aztec) diamond graphs}, which are the duals
of the so-called Aztec diamonds, which were first studied in detail
in~\cite{EKLP}.  A diamond graph of width $2k-1$ is obtained by taking
consecutive rows of squares, of length $1,3, \ldots , 2k-3, 2k-1,
2k-3, \ldots , 3,1$ and stacking them from top to bottom,
with the middle squares in all the rows lining up vertically,
as illustrated by Figure~\ref{fig:aztec}.  
Let $A$ be a diamond graph of width $2k-1$. 
Let $A_N$ be the diamond graph of width $2k-3$ obtained by deleting 
the leftmost and rightmost squares of $A$ as well as the two lowest squares 
of each of the remaining $2k-3$ columns of $A$.
We call $A_N$ the North sub-diamond of $A$. Define similarly 
the South, West and East sub-diamonds of $A$, denoted by $A_S, A_W$ and
$A_E$. Finally, let $A_C$ be the central sub-diamond of $A$ of width
$2k-5$ (Figure~\ref{fig:sub-diamonds}).
The following result is a
reformulation of Kuo's condensation theorem for Aztec
diamond graphs~\cite{Kuo}.

\begin{figure}[htb]
\begin{center}
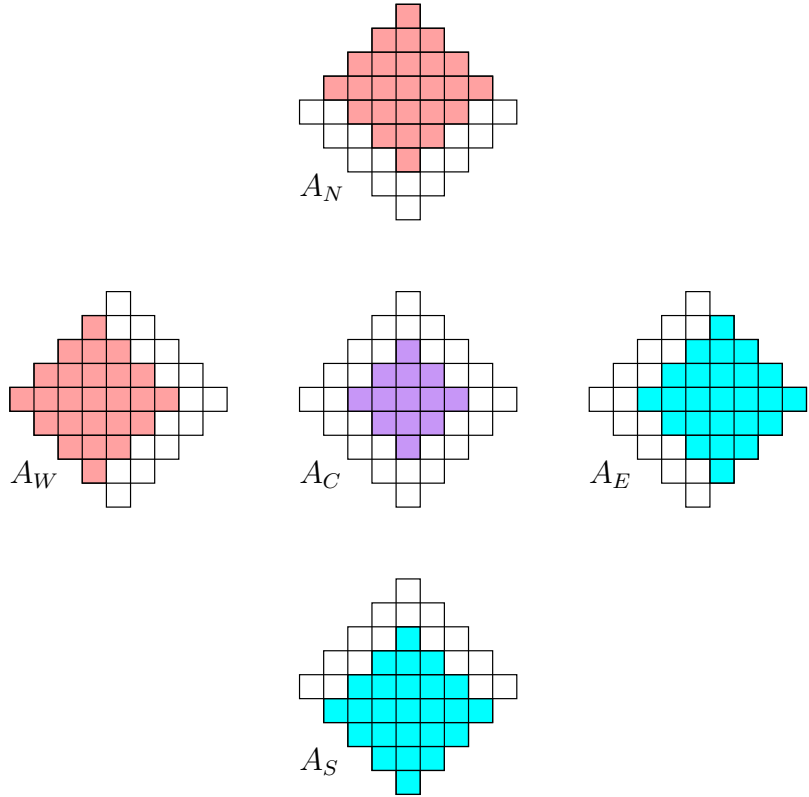
\end{center}
\caption{The five sub-diamonds of a diamond graph of width $9$.}
\label{fig:sub-diamonds}
\end{figure}

\begin{Theorem}[{\bf Condensation for diamonds graphs}] 
\label{thm:kuo}
The matching polynomial of a diamond graph $A$ is related to the matching
polynomials of its 
 sub-diamonds by
$$
M(A) M(A_C)= ns M(A_W) M(A_E) + ew M(A_N) M(A_S),
$$
where $n,s,w,$ and $e$ denote respectively the top (resp.~bottom,
westmost, eastmost) edge of $A$ (see Figure~{\it\ref{fig:aztec}}).
\end{Theorem}
In particular, if $a(n)$ (with $n \geq 2$) denotes the matching number 
of a diamond graph of width $2n-3$, then
$$
a(n)a(n-2)=2a(n-1)^2
$$
for all $n \geq 2$, provided we adopt the initial conditions $a(0)=a(1)=1$. 
This shows that $a(n)$ is 
the three-term Gale-Robinson sequence associated with $i=j=k=\ell=1$,
and implies $a(n)=2^{n\choose 2}$. 

The condensation theorem we shall prove for pinecones 
appears as a generalization of the condensation theorem for diamond
graphs. But it can 
actually also be seen as a \emph{specialization} of it, and this is the
point of view we adopt in this paper. The key idea is to forbid
certain edges in the matchings.
\begin{coro}
\label{coro:G}
Let $A$ be a diamond graph, and let $G$ be a spanning subgraph of
$A$, containing the edges $n,s,w$ and $e$. Let $G_N= G\cap A_N$, and
define $G_S, G_W, G_E$ and $G_C$ similarly. Then
$$
M(G) M(G_C) = ns M(G_W) M(G_E) + ew M(G_N) M(G_S).
$$
\end{coro}
\begin{proof} Since $G$ is a spanning subgraph of $A$, every perfect
matching of $G$ is a perfect matching of $A$. Hence the matching
polynomial $M(G)$ is simply obtained by setting
$a=0$ in $M(A)$, for every edge $a$ that belongs to $A$ but not to
$G$.
The same property relates $M(G_N)$ and $M(A_N)$, and so on. 
Consequently, Corollary~\ref{coro:G} is simply obtained by setting
$a=0$ in Theorem~\ref{thm:kuo}, for every edge $a$ that belongs to $A$
but not to $G$.
\end{proof}

\subsection{Pinecones: definitions}\label{sec:pinecones-def}
A \emph{standard pinecone} of width $2k-1$ 
is a subgraph $P=(V,E)$ of the square lattice
satisfying the three following conditions, illustrated by
Figure~\ref{fig:pinecone}.$a$:
\begin{itemize}
\item[$1.$] The horizontal edges form $i+j+1$ segments \emm of odd length,,
   starting from the points $(0,1), (1,2) \ldots , (i-1,i)$ and $(0,0), (1,-1),
   \ldots , (-j,j)$, for some $i \geq 1$, $j \ge 0$.
   Moreover, if $L_m$ denotes the length of the segment lying at
   ordinate $m$, then 
$$
L_{-j} < \cdots < L_{-1} < L_0 = 2k-1 = L_1 > L_2 > \cdots > L_i.
$$
\item[$2.$]  The set of vertices $V$ is the set of vertices of the square
   lattice that are incident to the above horizontal edges.
\item[$3.$]  Let $e=\{(a,b), (a,b+1)\}$ be a vertical edge of the square lattice
   joining two vertices of $V$. If $a + b$ is even, then $e$
   belongs to the set of edges $E$, and we say that $e$ is an \emm even,
   edge of $P$. Otherwise, $e$ may belong to
   $E$, or not (Figure~\ref{fig:pinecone}.$a$),
   and we call $e$ a (present or absent) \emm odd, edge.
\end{itemize}
The leftmost vertices of a standard pinecone are always $(0,0)$ and $(0,1)$.
However, sometimes it is convenient to consider graphs 
obtained by shifting such a graph to a different
location in the two-dimensional lattice.
We will call such a graph a {\it transplanted pinecone\/}.
In a transplanted pinecone, the leftmost vertices
are $(a,b)$ and $(a,b+1)$, where $a+b$ is even.
In some cases, where the distinction between
standard and transplanted pinecones is not relevant
or where we think the context makes it clear
which sort of pinecone we intend,
we omit the modifier and simply use the word ``pinecone''.

\begin{figure}[htb]
\begin{center}
\scalebox{0.9}{\input{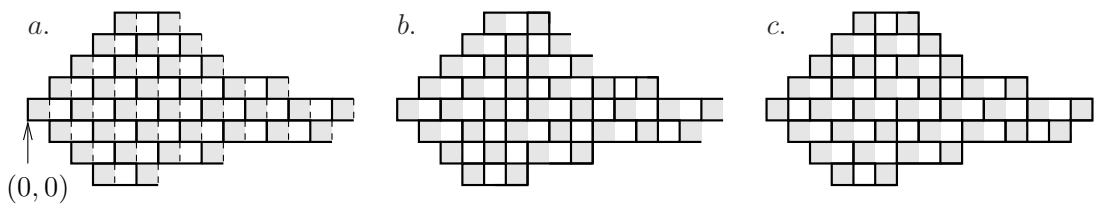}}
\end{center}
\caption{Some pinecones of width $15$. $a$. The dashed edges may
 belong to the graph, or not.  $b$. A pinecone. $c$. A closed pinecone.}
\label{fig:pinecone}
\end{figure}

Figures~\ref{fig:pinecone}.$b$ and~\ref{fig:pinecone}.$c$ show two 
specific ways to make the choices indicated in Figure~\ref{fig:pinecone}.$a$
and obtain a pinecone of width 15.
The pinecone of Figure~\ref{fig:pinecone}.$c$ is \emph{closed}, 
meaning that it contains no vertex of degree $1$.
(Such a vertex can only occur at the right border of the pinecone,
and occurs if the rightmost vertex of some horizontal segment
does not belong to a vertical edge,
as shown in Figure~\ref{fig:pinecone}.$b$.)
An Aztec diamond graph is an example of a closed pinecone. 
Let us color the cells of the square lattice black and white
in such a way that the cell containing the vertices $(0,0)$ and
$(1,1)$ is black. 
The \emm faces,  of the pinecone are
the finite connected components of the complement of the graph in $\rs^2$.
The faces of a pinecone $P$ are of three types: 
black squares, white squares, and horizontal rectangles
consisting of a black cell to the left and a white cell to the right. 
We insist on the distinction between a \emph{cell} (of the
underlying square lattice) and a \emph{square} (a face of $P$
that has 4 edges).  For instance, the longest row of a pinecone of width
$2k-1$ contains exactly $2k-1$ cells, but may contain no square at all.
Denoting by $\ell$ (resp.~$r$) the leftmost (resp.~rightmost)
cell of the longest row of $P$, we say that $P$ is \emph{rooted} on
$(\ell,r)$.  (If $P$ is standard, then $\ell$ is the cell
with $(0,0)$ as its lower-left corner.)

It is easy to see that a pinecone $P$ is closed if and only if the rightmost
finite face of each row is a black square. In this case, the rightmost black 
square in each row is also the rightmost cell of the row. If moreover $P$ is
standard, it is completely determined by the position of its black squares. 
Equivalently, it is completely determined by the position of its  
odd vertical edges.
Conversely, consider any finite set $S$ of black squares
whose lower-left vertices lie in the 90 degree wedge
bounded by the rays $y=x>0$ and $-y=x>0$.  
Assume that $S$ is \emm
monotone,, in the following sense: 
the rows that contain at least one square of $S$ are
consecutive (say from row $-j$ to row $i$) and
for $m>0$ (resp.~$m<0$), the rightmost black square in row $m$ 
occurs to the left of the rightmost black square 
in row $m-1$ (resp.~$m+1$).
Then is a (unique) closed standard pinecone whose set of black
squares is $S$.

We shall often consider the empty graph as a particular closed pinecone
(associated with the empty set of black squares).
The empty graph has one perfect matching, of weight 1.

\subsection{The core of a pinecone}
When a pinecone $P$ is not closed,
some of the edges of $P$ cannot belong to any perfect matching of $P$.
Specifically, if $v$ is a vertex of degree 1 in $P$,
then in any perfect matching of $P$,
$v$ must be matched with the vertex to its left (call it $u$),
so that $u$ cannot be matched with any of its other neighbors.
Indeed, there can be a chain reaction
whereby a forced edge, in causing other edges to be forbidden,
leads to new vertices of degree 1,
continuing the process of forcing and forbidding other edges.
An example of this is shown in Figure~\ref{fig:pinecone-shelling}.
The left half of the picture shows a non-closed pinecone $P$,
and the right half of the picture shows a closed
sub-pinecone $\bar P$ of $P$ along with a set of isolated edges.
The reader can check
(starting from the rightmost frontier of $P$
and working systematically leftward)
that each of the isolated edges is a forced edge
(that is, it must be contained in every perfect matching of $P$),
so that a perfect matching of $P$
is nothing other than a perfect matching of $\bar P$
together with the set of isolated edges shown at right.
In this subsection, we will give a systematic way of
reducing a pinecone $P$ to a smaller closed pinecone
by pruning away some forced and forbidden edges.

\begin{figure}[htb]
\begin{center}
\input{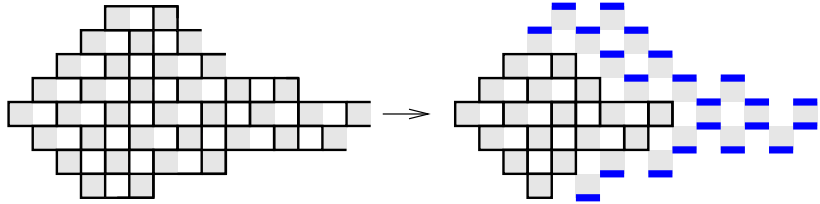}
\end{center}
\caption{From a pinecone $P$ to its core $\bar P$.}
\label{fig:pinecone-shelling}
\end{figure}

It can easily be checked that the union or intersection
of two standard pinecones is a standard pinecone,
and that the union or intersection
of two closed standard pinecones is a closed standard pinecone.
It follows that, if $P$ is a standard pinecone,
there exists a largest closed standard sub-pinecone of $P$,
namely, the union of all the closed standard sub-pinecones of $P$.
We call this the \emm core, of $P$ and denote it by $\bar P$.
(If $P$ is not a standard pinecone but a transplanted pinecone
rooted at the cell with lower-left corner $(a,b)$,
we define $P_0$ as the standard pinecone
obtained by translating $P$ by $(-a,-b)$,
and we define the core of $P$ as the core of $P_0$ translated by $(a,b)$.
However, for the rest of this section we will
restrict attention to standard pinecones.)

Here is an alternative (more constructive and less abstract)
approach to defining the core.  Let $P$ be a standard pinecone.
Let $b_0$ be the rightmost black square in row 0 of $P$,
let $b_1$ be the rightmost black square in row 1 of $P$
that lies strictly to the left of $b_0$,
let $b_2$ be the rightmost black square in row 2 of $P$
that lies strictly to the left of $b_1$, and so on (proceeding upwards);
likewise, let $b_{-1}$ be the rightmost black square in row $-1$ of $P$
that lies strictly to the left of $b_0$,
and so on (proceeding downwards).
If at some point there is no black square that satisfies the requirement,
we leave $b_m$ undefined.
Consider all the faces of $P$ that lie in the same row as,
and lie weakly to the left of, one of one of the $b_k$'s.
This set of faces gives a closed pinecone $\tilde P$.
At the same time, it is clear that
any closed sub-pinecone $Q$ of $P$ must be a sub-pinecone of $\tilde P$.
For, the rightmost black square in row 0 of $Q$
can be no farther to the right than $b_0$,
which implies that the rightmost black square in row 1 of $Q$
can be no farther to the right than $b_1$, etc.;
and likewise for the bottom half of $Q$.
Hence the sub-pinecone $\tilde P$ we have constructed
is none other than the core of $P$ as defined above.

If $P$ is closed, then $\bar P = P$.
Note that a closed pinecone always admits
two particularly simple perfect matchings:
one consisting entirely of horizontal edges,
and the other consisting of the leftmost and rightmost vertical edges
in each row (and no other vertical edges)
along with some horizontal edges (Figure~\ref{fig:special-matchings}).
In particular, the rightmost vertical edges of a closed pinecone are
never forced nor forbidden.

\begin{figure}[htb]
\begin{center}
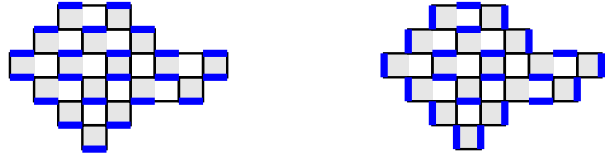
\end{center}
\caption{Two particularly simple matchings of a closed pinecone.}
\label{fig:special-matchings}
\end{figure}

Let $P$ be a pinecone with core $\bar P$.
There is a unique perfect matching of $P \setminus {\bar P}$
consisting exclusively of horizontal edges (see
Figure~\ref{fig:pinecone-shelling});
let $H$ be the edge set of this perfect matching.
Every perfect matching of $\bar P$ can be extended to a perfect matching of $P$
by adjoining the edges in $H$, so $m(P) \geq m({\bar P})$.
We now show that every perfect matching of $P$
is obtained from a perfect matching of $\bar P$ in this way.

\begin{propo}
\label{core-recursive}
Let $P$ be a pinecone with core $\bar P$.
Then $m({\bar P}) = m(P)$.
\end{propo}
\begin{proof}
We will prove this claim by using a procedure 
that reduces a sub-pinecone $Q$ of $P$
to a smaller sub-pinecone with the same matching number.
Let $Q$ be a sub-pinecone of $P$ whose core coincides with $\bar P$. 
If $Q$ is not closed, 
then there must be at least one vertex of degree 1 
along the right boundary of $Q$.
Let $v=(a,b)$ be one of the the rightmost vertices of degree $1$ in $Q$. 
Then $v$ is the rightmost vertex in one of the rows of $Q$.  
Assume for the moment that $v$ lies strictly above the longest row of $Q$ 
(that is, $b > 1$). 
See the top part of Figure~\ref{fig:new-sequence} for an 
illustration of the following argument.
Let $u$ be the vertex to the left of $v$.
Then the edge joining $u$ and $v$ is forced to belong to every
perfect matching of $Q$,
while every other edge containing $u$ 
is forbidden from belonging to any perfect matching of $Q$.
Hence the graph $Q'$ obtained from $Q$
by deleting $u$, $v$, and every edge incident with $u$ or $v$
has the same matching number as $Q$.
Furthermore, $Q'$ is a pinecone, 
unless the vertex $v_1=(a-1,b+1)$ belongs to $Q$. 
In this case, $v_1$ has degree 1 in $Q'$.
Let $i$ be the largest integer such that 
$v_j=(a-j,b+j)$ belongs to $Q$ for all $0 \leq j \le i$.
Applying the deletion procedure to the vertices 
$v=v_0, v_1, \ldots , v_i$ (in this order) yields a pinecone $Q^*$.
Assume now that $b=1$.
Applying the deletion procedure to all the vertices of $Q$ 
of the form $(a-j,1+j)$ or $(a-j,-j)$
yields again a pinecone $Q^*$ 
(see Figure~\ref{fig:new-sequence}, bottom). 
By symmetry, we have covered all possible values of $b$. 

\begin{figure}[htb]
\begin{center}
\scalebox{0.85}{\input{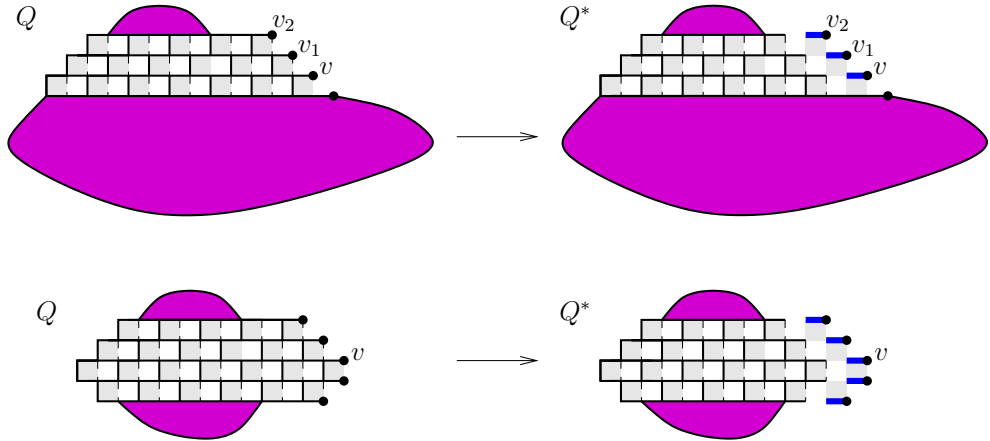}}
\end{center}
\caption{Some sequences of edge-deletions starting and ending with a pinecone.}
\label{fig:new-sequence}
\end{figure}

Observe that $m(Q)=m(Q^*)$.
Additionally, we can check that the core of $Q^*$ is $\bar P$.
The only thing we might worry about is that
in passing from $Q$ to $Q^*$,
we removed some edges that belong to $\bar P$. 
The examination of Figure~\ref{fig:pinecone-shelling} shows that 
we would have, in particular, 
removed the rightmost vertical edge is some row of $\bar P$.
However, this cannot happen, because the removed edges
were all forced or forbidden,
whereas the rightmost edges of $\bar P$
are neither forced nor forbidden (Figure~\ref{fig:special-matchings}).

To prove that $m({\bar P}) = m(P)$, take $Q=P$ 
and use the preceding operation repeatedly
to construct successively smaller graphs $Q^*$, $Q^{**}$, \dots
such that $m(P) = m(Q) = m(Q^*) = m(Q^{**}) = \cdots$
and $\bar P = \bar Q = 
\overline {Q^*} = \overline {Q^{**}}= \cdots\:$.
Eventually we arrive at a closed sub-pinecone of $P$
whose core is $\bar P$;
that is, we arrive at $\bar P$ itself.
And since each step of our construction
preserves $m(Q)$,
we conclude that $m({\bar P}) = m(P)$, as claimed.
\end{proof}

\subsection{A condensation theorem for closed pinecones}
\label{section:condensation-pinecones}
Let $P$ be a closed pinecone,
with longest row consisting of $2n+1$ squares.
Let $A$ be the smallest diamond graph containing $P$ 
(the longest row of $A$ contains exactly $2n+1$ cells).
Let $G$ denote the spanning subgraph of $A$
whose edge-set consists of all edges of $P$, 
all horizontal edges of $A$, 
and all \emm even, vertical edges of $A$ 
(Figure~\ref{fig:completion-pinecone}).  
Observe that $G$ is a pinecone. 
Moreover, among the spanning subgraphs of $A$ that are
pinecones and contain $P$, 
$G$ has strictly fewer edges than the others.  
Since no odd vertical edge is added, 
$P$ is actually the core of the pinecone $G$.

\begin{figure}[htb]
\begin{center}
\input{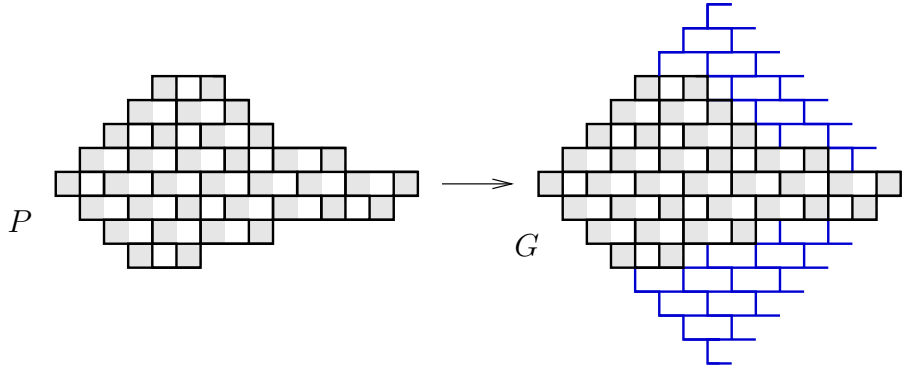}
\end{center}
\caption{Completing a pinecone $P$ into a spanning pinecone of an
Aztec diamond graph.}
\label{fig:completion-pinecone}
\end{figure}

Let us now use the notation of Corollary~\ref{coro:G}. That is,
$G_N=G\cap A_N$, and so on.
Then $G_N, G_S, G_W, G_E$ and $G_C$ are (standard or transplanted) pinecones.
Let $P^N$, $P^S$, $P^W$, $P^E$ and $P^C$ denote their respective cores.
(These are not to be confused with $P_N$, etc.,
which are the intersections of $P$ with $A_N$, etc.) 
We will often call $P^N$, $P^S$, $P^W$, $P^E$ and $P^C$ ``the five
sub-pinecones'' of $P$, even though, strictly speaking, $P$ admits
other sub-pinecones.
An example is given in Figure~\ref{fig:sub-pinecones}.  Let
$\ell_0$ (resp.~$r_0$) be the leftmost (resp.~rightmost)
cell of the longest row $R_0$ of $P$. Similarly, let 
$r_1$ (resp.~$r_{-1}$)
denote the rightmost cell of the row just above (resp.~below)
$R_0$. Observe that the cells $r_0, r_1$ and $r_{-1}$ correspond to
black squares of $P$. Finally, let $\ell'_0$ be the black \emph{cell} of
$R_0$ following $\ell_0$, and let $r'_0$ the black \emph{square} of $R_0$
preceding $r_0$ (if it exists).  
In light of the basic properties of the core (both the abstract
definition and the algorithmic construction), 
we can give the following alternative description of 
the five   sub-pinecones of $P$.

\begin{figure}[htb]
\begin{center}
\scalebox{0.9}{\input{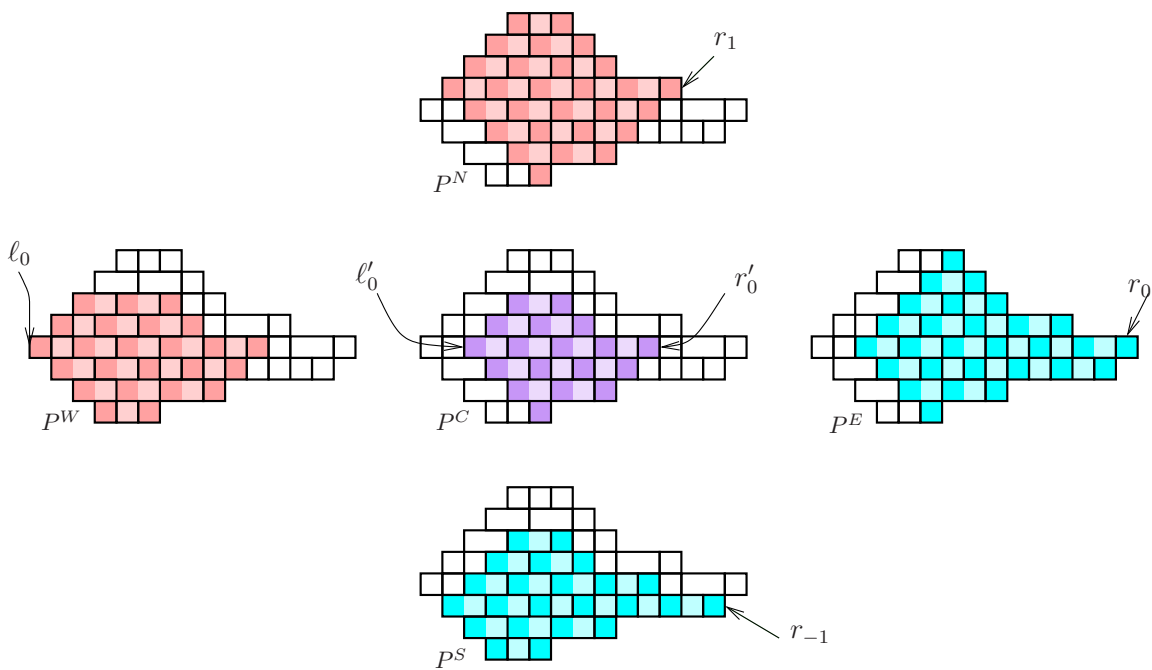}}
\end{center}
\caption{The five sub-pinecones of a pinecone $P$.}
\label{fig:sub-pinecones}
\end{figure}

\begin{propo}
\label{big-from-small}
Let $P$ be a closed pinecone. With the above notation, $P^N$ (resp.~$P^S$) is
the largest closed sub-pinecone of $P$ whose 
rightmost cell is $r_1$ (resp.~$r_{-1}$). Similarly, $P^W$ (resp.~$P^E$) is the
largest closed sub-pinecone whose 
rightmost (resp.~leftmost) cell is $r'_0$ (resp.~$\ell'_0$).
Finally, $P^C$ is the largest closed sub-pinecone rooted on $(\ell'_0, r'_0)$.
\end{propo}
This proposition implies that a pinecone $P$ that is neither empty,
nor reduced to a black square can be reconstructed from its four
main sub-pinecones $P^N$, $P^S$, $P^E$ and $P^W$. Indeed, the part of
$P$ located strictly above its longest row coincides 
with the top part of $P^N$. More precisely, row $r$ of $P$, with $r>0$
coincides with row $r-1$ of $P^N$. 
Similarly, 
for $r<0$, row $r$ of $P$ coincides with row $r+1$ of $P^S$. It thus
remains to determine the longest row of $P$.
This row is obtained by adding a $2$-by-$1$ rectangle\footnote{This
 rectangle is actually only useful if $P^W$ is empty or reduced to a
 single black square.} to the left of the longest row of $P^E$, and
then superimposing the longest row of $P^W$.

Let us now apply Corollary~\ref{coro:G} 
to the graph $G$ obtained by completing $P$ into a
spanning pinecone of $A$
(Figure~\ref{fig:completion-pinecone}). By Proposition~\ref{core-recursive},
since $P$ is the 
core of $G$, $m(G)=m(P)$, and similar identities relate the matching
numbers of $G_N$ and $P^N$, etc. 
\begin{Theorem}[{\bf Condensation for closed pinecones}]
\label{thm:condensation}
The matching number of a closed pinecone $P$ is related to the matching
number of its 
 closed sub-pinecones by
$$
m(P) m(P^C) = m(P^W) m(P^E) +
m(P^N) m(P^S). 
$$ 
\end{Theorem}
\noindent 
We will state in Section~\ref{sec:bivariate} a more general
condensation result
dealing with the matching polynomial, rather than the matchings
number, of closed pinecones (Theorem~\ref{thm:condensation-complete}).

\section{Pinecones for the Gale-Robinson sequences}\label{sec:pinecones-GR}

The pinecones introduced in the previous section 
generalize Aztec diamond graphs.
The number of perfect matchings of the diamond graph of 
width $2n-3$ is the $n$th term in the recurrence
\begin{eqnarray*}
a(n)a(n-2) = a(n-1)a(n-1) + a(n-1)a(n-1),
\end{eqnarray*}
with initial conditions $a(0)=a(1)=1$.
More generally, the three-term Gale-Robinson sequences are
governed by recurrences of the form 
\beq
\label{GR-rec}
a(n)a(n-m) = a(n-i)a(n-j) + a(n-k)a(n-\ell),
\eeq
with initial conditions $a(n)=1$ for $n=0, 1, \ldots, m-1$.  Here, $i,
j, k$ and 
$\ell$ are positive integers such that $i+j=k+\ell=m$,
and we adopt the following (important) convention
$$
j=\min\{i,j,k,\ell\}.
$$

Our purpose in this section is to construct a sequence of (closed) pinecones
$(P(n))_{n\ge 0} \equiv \left(P(n;i,j,k,\ell)\right)_{n\geq 0}$ 
for each set of parameters $\{i,j,k,\ell\}$ such that $i+j=k+\ell=m$,
and to show that the matching numbers of the pinecones in our 
sequence satisfy the corresponding Gale-Robinson recurrence.
More specifically, our family of graphs will be constructed
in such a way that
\begin{itemize}
\item $P(n)^C$ is $P(n-m)$ transplanted to $(2,0)$ 
(that is, shifted two steps to the right),
\item $P(n)^W$ is $P(n-i)$,
\item $P(n)^E$ is $P(n-j)$ transplanted to $(2,0)$,
\item $P(n)^N$ is $P(n-k)$ transplanted to $(1,1)$, and
\item $P(n)^S$ is $P(n-\ell)$ transplanted to $(1,-1)$.
\end{itemize}
In our construction, we use the fact that a closed pinecone is
completely determined by its set of odd vertical edges, that is,
vertical edges of the form $\{(a,b),(a,b+1)\}$ where $a+b$ is odd.
We introduce two functions, an upper function $U$ and 
a lower function $L$, which will be used to determine
the positions of the odd vertical edges 
in the (closed) pinecone $P(n;i,j,k,\ell)$: for $r \ge 0$, let 
\beq
\label{UL-def}
\begin{array}{lll}
U(n,r,c) &= &2c + r -3 - 2 \left\lfloor{\displaystyle \frac{mc+kr+
    i-n-1}{j}}\right\rfloor , \\
\\
L(n,r,c) &=& 2c + r -3 - 2 
\left\lfloor{\displaystyle \frac{mc+\ell r+ i-n-1}{j}}\right\rfloor.
\end{array}
\eeq
Observe that the parameters $k$ and $\ell$ play symmetric
roles. Also, $U(n,0,c)=L(n,0,c)$. The 
function $U$ will describe the upper part of the pinecone, while $L$
will describe its lower part. Recall that, by convention, the longest 
row of a standard pinecone
is row $0$ and its South-West corner lies at coordinates $(0,0)$, 
as shown in Figure~\ref{fig:pinecone}.

To locate the vertical odd edges in row $r \geq 0$, calculate the values 
$U(n,r,c)$ for $c=0,1,\ldots$  This will be a (strictly) decreasing sequence,
since $m \geq 2j$ (recall that $i+j=m$ and $j \leq i$).
Retain those values $U(n,r,c)$ that are larger than $r$, 
and place a vertical edge in the $r$th row at abscissa $U(n,r,c)$, 
that is, an edge connecting $(U(n,r,c),r)$ and $(U(n,r,c),r+1)$
(an odd edge, since $U(n,r,c)+r$ is odd).
The first row not containing such an edge (and therefore not
included in the pinecone) is the first one for which
$U(n,r,0) < r$. Observe that $U(n,r,0)-r$ is a decreasing function of
$r$ (since $j\le k$).  This property guarantees that if the $r$th row is empty,
then all higher rows are empty too. It also implies that the rightmost
vertical edge in row $r$ (which is located at abscissa $U(n,r,0)$)
lies to the right of the rightmost vertical edge in row $r+1$.
(To see this, note that $U(n,r,0)-r$ is always an odd number.
So the inequality $U(n,r+1,0)-(r+1) < U(n,r,0)-r$
implies $U(n,r+1,0)-(r+1) \leq U(n,r,0)-r-2$,
or $U(n,r+1,0) < U(n,r,0)$.
That is, the set of odd edges (equivalently, of black squares) given 
by Formula~(\ref{UL-def})
satisfies the ``top part'' of the monotonicity condition described at the end of
Section~\ref{sec:pinecones-def}: the rightmost odd edge in row $r>0$,
if it exists, lies to the left of the rightmost odd edge in row $r-1$.

Similarly, to locate the edges in row $-r \leq 0$, calculate the
values of $L(n,r,0) > L(n,r,1) > \cdots$ and retain those larger
than $r$.  For each, place a vertical edge in row $-r$ at abscissa
$L(n,r,c)$, that is, connecting 
$(L(n,r,c),-r)$ and $(L(n,r,c),-r+1)$.
Observe that $L(n,0,c)=U(n,0,c)$, so that the collection of odd
vertical edges in row $0$ is the same whether it is determined from $U$
or from $L$.

\medskip
The monotonicity properties satisfied by the positions of the
odd edges imply that there exists a unique standard closed pinecone
whose set of odd vertical edges coincides with the set we have
constructed via the functions $U$ and $L$. 
To obtain this pinecone,
draw horizontal edges from $(r,r+1)$ to $(U(n,r,0),r+1)$ and
from $(-r,-r)$ to $(L(n,r,0),-r)$ for all $r \geq 0$.  Finally, place
all the appropriate even vertical edges.  Since these steps are
so routine, we regard the pinecone as fully described once the set
of odd vertical edges has been specified.  This point of view 
simplifies the exposition.

Observe that $P(n)$ is empty if and only if $U(n,0,0)<0$, 
which is equivalent to $U(n,0,0) \le -1$
(since $U(n,0,0)$ is odd),
which is easily seen to be equivalent to $n<m$
(using the fact that $m=i+j$).
\smallskip

\noindent
{\bf Example.}
Take $(i,j,k,\ell)=(5,2,3,4)$ and determine
$P(12)$. The above definition of $U$ and $L$ specializes to
\begin{eqnarray*}
U(n,r,c) &=& 2c + r -3 - 2 \left\lfloor{\frac{7c+3r-8}{2}}\right\rfloor , \\
L(n,r,c) &=& 2c + r -3 - 2 \left\lfloor{\frac{7c+4r-8}{2}}\right\rfloor.
\end{eqnarray*}
In row $0$, we find odd edges with lower vertices $(5,0)$ and
$(1,0)$.  In row $1$, there is one odd edge at $(4,1)$.  This is the
top row of the diagram because $U(12,2,0)=1<2$.  Turning to the
lower portion of the diagram, there is one odd edge with lower
vertex $(2,-1)$ and none in row $-2$ or below.  Completing the 
diagram is now routine, and gives the pinecone $P(12)$ which is shown in
Figure~\ref{fig:construction-example}, together with its 14 perfect
matchings. Accordingly, the Gale-Robinson sequence $a(n)$ associated with
$(5,2,3,4)$ satisfies $a(12)=14$.

A larger example is presented after Corollary~\ref{GR-coro}.
\begin{figure}[hbt]
\begin{center}
\input{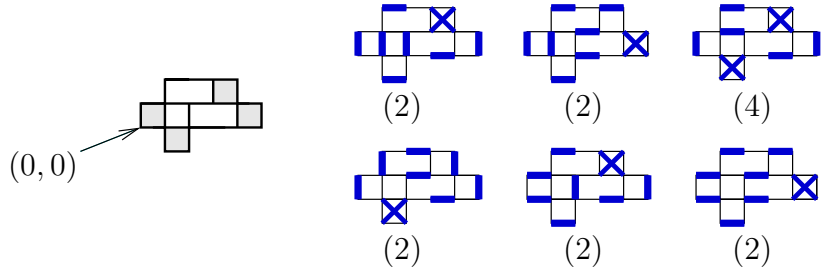}
\end{center}
\caption{The pinecone $P(12; 5, 2, 3,4)$, with black squares
  indicated, and its 14 perfect matchings 
  (a cross stands for any of the two matchings of a square).}
\label{fig:construction-example}
\end{figure}

The pinecones based on the functions $U$ and $L$ 
satisfy a remarkable property: the odd edges (or,
equivalently, 
the black squares) in rows $r$ and $r+1$ are \emm interleaved,. 
That is, between two black squares in row $r\ge 0$, there is a black square
in row $r+1$, and similarly, between two black squares in row $r\le
0$, there is a black square in row $r-1$. 
This can be checked on the small example of
Figure~\ref{fig:construction-example}, but is more visible on the
bigger example of Figure~\ref{big-example}.
\begin{Lemma}[{\bf The interleaving property}]
For all values of $n, r$ and $c$, the functions $U$ and $L$ defined
by~{\em\Ref{UL-def}} satisfy 
$$
U(n,r,c+1)+1\le U(n,r+1,c)\le U(n,r,c)-1
$$
and
$$
L(n,r,c+1)+1\le L(n,r+1,c)\le L(n,r,c)-1.
$$
\end{Lemma}
\begin{proof} 
We have
$$
U(n,r+1,c)-U(n,r,c+1) = 2 \left\lfloor
\frac{mc+kr+i-n-1+m}{j} \right\rfloor - 2 \left\lfloor
\frac{mc+kr+i-n-1+k}{j} \right\rfloor -1.
$$
But
$$
\frac{mc+kr+i-n-1+m}{j}-\frac{mc+kr+i-n-1+k}{j}= \frac \ell j \ge 1,
$$
so that the two floors occurring in the above identity differ by 1 at
least. Consequently,
$$
U(n,r+1,c)-U(n,r,c+1) \ge 2 -1=1.
$$
The three other inequalities are proved in a similar manner.
\end{proof}

\medskip
We now wish to apply the condensation theorem
(Theorem~\ref{thm:condensation}) to the 
pinecones $P(n)$ we have just defined. 
Using the notation of Theorem~\ref{thm:condensation}, we will
verify that, up to translation,   
$P(n)^W = P(n-i)$, $P(n)^E = P(n-j)$, $P(n)^N = P(n-k)$, 
$P(n)^S = P(n-\ell )$ and $P(n)^C = P(n-m)$.  These equivalences
will follow from the interleaving property and the following algebraic
equalities.

\begin{Lemma} 
\label{lemma:arithmetic}
For any choice of parameters $(i,j,k,\ell)$, the functions
  $U$ and $L$ defined by~{\em\Ref{UL-def}} satisfy: 
$$\begin{array}{rcllrcl}
U(n-i,r,c-1) & = &  U(n,r,c),      &  &  L(n-i,r,c-1) & = & L(n,r,c), \\
U(n-j,r,c) & = & U(n,r,c) - 2,     &  &  L(n-j,r,c) & = & L(n,r,c) - 2, \\
U(n-k,r-1,c) & = & U(n,r,c) - 1,   & \hskip 4mm &  L(n-\ell,r-1,c) & = & L(n,r,c) - 1,  \\
U(n-\ell,r+1,c-1) & = & U(n,r,c) - 1, &  &  L(n-k,r+1,c-1) & = & L(n,r,c) - 1,\\ 
U(n-m,r,c-1) & = & U(n,r,c) - 2,   &  &  L(n-m,r,c-1) & = & L(n,r,c) - 2. \\
\end{array}$$
\end{Lemma}

\noindent
\begin{proof}
The $L$-identities are symmetric to the
$U$-identities upon exchanging $k$ and $\ell$, so that there are
really 5 identities to prove.
These can all be verified by routine algebraic manipulations. Let us check
for instance the fourth identity satisfied by $U$:
\begin{eqnarray*}
U(n-\ell,r+1,c-1) &=& 2(c-1) + (r+1) -3 - 2
\left\lfloor{\frac{m(c-1)+k(r+1) + i-(n-\ell )-1}{j}}\right\rfloor \\
&=& 2c + r -4 - 2 
    \left\lfloor{\frac{mc +kr +i -n -1 -m+k+\ell}{j}}\right\rfloor \\
&=& 
2c + r -4 - 2 \left\lfloor{\frac{mc +kr +i -n -1}{j}}\right\rfloor
\hskip 12mm \hbox{since } m=k+\ell\\
&=& U(n,r,c) - 1.
\end{eqnarray*}
We leave it to the reader to verify the remaining 4 identities. 
\end{proof}

We now check that these identities imply that the pinecones are related
to one another as claimed.

\begin{propo} \label{lemma:GRpine}  
Let $P(n)\equiv P(n;i,j,k,\ell)$ be the sequence of pinecones
associated with the parameters $(i,j,k,\ell)$. Then for $n \ge m$, the
five closed sub-pinecones of $P(n)$ satisfy
\begin{eqnarray*}
P(n)^W = P(n-i), &\quad & P(n)^E = P(n-j), \\ 
P(n)^N = P(n-k),& \quad& P(n)^S = P(n-\ell),
\end{eqnarray*}
and
$$
 P(n)^C = P(n-m).
$$
These identities hold up to a translation.
\end{propo}
\begin{proof}
Begin by checking that $P(n)^W=P(n-i)$.  Using the description of
$P(n)^W$ given in Proposition~\ref{big-from-small}, and the fact that 
the black squares of $P(n)$ are interleaved, we see that the
odd vertical edges in $P(n)^W$  
are those of $P(n)$, except that the rightmost odd edge in each
row has been removed.  (If this was the only odd edge in the row, then
the entire row disappears.)  Therefore $P(n)^W$ can be constructed
by following the construction for $P(n)$, but beginning with 
$c=1,2,\ldots$ instead of $c=0,1,\ldots$. 
This means that in row $r\geq 0$ of $P(n)^W$, odd edges appear at 
positions $U(n,r,1),U(n,r,2),\ldots$ as long as these values continue 
to exceed $r$.  (Similarly in rows $r\leq 0$, using $L$ instead of $U$.)  

Let us now compare this with $P(n-i)$. In row $r\geq 0$ of $P(n-i)$,
odd edges appear in positions 
$U(n-i,r,0),U(n-i,r,1),\ldots$, as long as these values continue to
exceed $r$.  However we showed in
Lemma~\ref{lemma:arithmetic} that  
$U(n-i,r,c-1)=U(n,r,c)$, so the sequence of odd edges in 
row $r$ is the same in $P(n)^W$ and in $P(n-i)$.  The situation is
similar in rows $r<0$ using the equality for $L$.  As we remarked 
above, a pinecone is determined by its odd edges (and the position of
its leftmost edge), so $P(n)^W=P(n-i)$.

\medskip
The other four equivalences are similar.  The only new development
is that instead of being positioned at the origin, the smaller
pinecones are now offset by one or two columns (in all four cases)
and possibly rows (in the case of $P(n)^N$ and $P(n)^S$).
We will look at $P(n)^S$ as an example, and let the reader supply
the details for the remaining three cases.

As noted after Proposition~\ref{big-from-small},  for $r \le 0$, row
$r$ of $P(n)^S$ coincides with row $r-1$ of $P(n)$ (see
Figure~\ref{fig:sub-pinecones} for an example).
 For $r>0$, the leftmost cell of row $r$ of $P(n)^S$ lies two steps
 to the right of the leftmost cell of row $r-1$ of $P(n)$. Moreover,
 the interleaving property implies, by induction on $r$, that the
 last (\emm i.e.,, rightmost) black square of row $r$ of $P(n)^C$ is
 the next-to last black square of row $r-1$ of $P(n)$.  
Thus the odd edges of $P(n)^S$ are located as follows: for rows
$-r$, with $r=1,2,\ldots$, in  
columns $L(n,r,0), L(n,r,1),\ldots$, as long as these numbers
continue to exceed $r$; and for rows $r=0,1,2,\ldots$, in columns
$U(n,r,1),U(n,r,2),\ldots$, as long as these numbers continue to
exceed $r+2$.

Let us now look at a copy of $P(n-\ell)$ positioned with its
origin at $(1,-1)$.  After this translation, the odd vertical
edges in rows $-r$, with $r=1,2,\ldots$ 
are located at abscissas 
$L(n-\ell,r-1,c)+1$, for $c\ge 0$ and as long as these numbers
continue to exceed $r$.  Lemma~\ref{lemma:arithmetic} then implies
that the bottom parts of $P(n)^S$ and of the translate of $P(n-\ell)$
coincide. 
After the translation, the odd vertical
edges of $P(n-\ell)$ lying in rows $r$, with $r=0, 1,2,\ldots$ are
located at abscissas   
$U(n-\ell,r+1,c)+1$, for $c\ge 0$ and as long as these numbers
continue to exceed $r+2$. Lemma~\ref{lemma:arithmetic} then implies
that the top parts of $P(n)^S$ and of the translate of $P(n-\ell)$ coincide.

\smallskip
This completes the analysis for $P(n)^S$; the verifications for
$P(n)^N$, $P(n)^E$ and $P(n)^C$ are similar 
(and even identical, up to symmetry, in the case of $P(n)^N$).
\end{proof}

\begin{figure}[thb]
\begin{center}
\scalebox{0.85}{\input{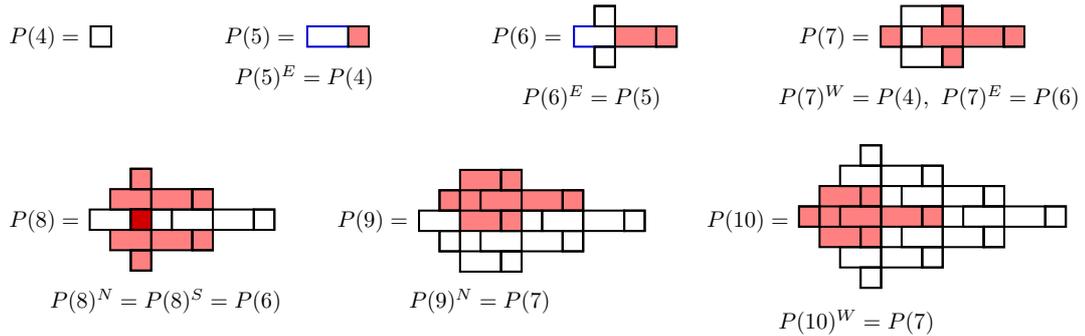}}
\end{center}
\caption{Recursive graphical construction of the pinecones associated with
the Somos-4 sequence.  At each stage, one (or two) of the components
that are superimposed to form the pinecone is highlighted.}
\label{fig:somos4-rec}
\end{figure}

\noindent{\bf Remark: a recursive construction of the pinecones
 $P(n)$.}
The above proposition, combined with
Proposition~\ref{big-from-small},  provides an alternative
way of constructing the sequence of pinecones $P(n)$ associated with a
given set of parameters $(i,j,k,\ell)$.
For $0 \leq n < m$, we put $P(n)$ equal to the empty graph
(which has one perfect matching),
and for $m \leq n < m+j$, we put $P(n)$ equal to the graph
with four vertices and four edges surrounding one square face
(which has 2 perfect matchings). 
Then, for $n\ge m+j$, it suffices to superimpose $P(n-i),
P(n-j), P(n-k)$ and $P(n-\ell)$, and add a $2$-by-$1$ rectangle to the left
of the longest row of $P(n-j)$. More precisely, the four above pinecones
must be positioned in such a way the leftmost cell of $P(n-i)$
(resp.~$P(n-j), P(n-k),P(n-\ell)$) has its South-West corner at
 $(0,0)$ (resp.~$(2,0)$, $(1,1)$, $(1,-1)$), while the
$2$-by-$1$ rectangle has its South-West corner at $(0,0)$. 
(This rectangle is only needed if $P(n-i)$ is empty
or consists of a single black square.
Typically this 2-by-1 rectangle comes for free as part of $P(n-i)$.
Note that we do not claim that this rectangle is a face of the pinecone;
the odd edge joining $(1,0)$ and $(1,1)$
will be present or absent in $P(n)$,
according to whether it is present or absent in $P(n-i)$.)
This gives a graphical, inductive way of constructing $P(n)$. 
This method is illustrated in 
Figure~\ref{fig:somos4-rec} by the case of the Somos-4 sequence, for which 
$$
a(n)a(n-4)=a(n-3)a(n-1)+a(n-2)^2.
$$
That is, $(i,j,k,\ell)=(3,1,2,2)$ and $m=4$.

\medskip

We can  now state our combinatorial interpretation of the Gale-Robinson numbers.
\begin{Theorem}
\label{thm:GR-pinecones}
Let $P(n)\equiv P(n;i,j,k,\ell)$ be the sequence of pinecones
associated with the parameters $(i,j,k,\ell)$. Let $a(n)$ denote the
number of perfect matchings of $P(n)$. Then $a(n)=1$ for $n<m$, and
for $n \ge m$, the sequence $a(n)$ satisfies
the following Gale-Robinson recurrence:
$$
a(n)a(n-m) = a(n-i)a(n-j)+a(n-k)a(n-\ell).
$$
\end{Theorem}
\begin{proof}
We have already observed that the pinecone $P(n)$ is 
empty for $n<m$.  Hence the initial conditions apply correctly.
Now for $n\ge m$, Theorem \ref{thm:condensation} states that
the matching matching of 
$P(n)$ is related to the matching numbers of its closed
sub-pinecones by
\begin{eqnarray*}
 m(P(n)) m(P(n)^C) = m(P(n)^W) m(P(n)^E) + m(P(n)^N) m(P(n)^S).
\end{eqnarray*}
Proposition~\ref{lemma:GRpine} then implies that 
$m(P(n)^C) = m(P(n-m))$, etc.  Therefore,
\begin{eqnarray*}
m(P(n))m(P(n-m)) = m(P(n-i))m(P(n-j)) +x
m(P(n-k))m(P(n-\ell)),
\end{eqnarray*}
which is the recurrence relation satisfied by $a(n)$.
\end{proof}

Before we study a specific example, let us state an obvious corollary
of Theorem~\ref{thm:GR-pinecones}.
\begin{coro} \label{GR-coro}
Let $i,j,k, \ell$ be positive integers such that $i+j=k+\ell=m$. 
The recurrence relation
$$
a(n)a(n-m) = a(n-i)a(n-j) + a(n-k)a(n-\ell),
$$
with initial conditions $a(n)=1$ for $n<m$, defines a sequence of
positive integers
\end{coro}

\medskip
\noindent{\bf Example.}
We give a specific example in the case where $(i,j,k,\ell)$ = $(6,2,5,3)$ 
and $n=25$.
We also show how to use the VAXmaple software package
(available at {\tt http://jamespropp.org/vaxmaple.c})
to compute the number of perfect matchings in the constructed pinecone,
which can be seen to be the $25$th term in the appropriate Gale-Robinson
sequence.

Considering first the upper portion of $P(n)$, we fix $r$ and
then consider the first few values of $U(n,r,c)$ as $c=0,1,2,3,\ldots$:
\begin{eqnarray*}
r = 0:&&  \{17, 11, 5, -1, \ldots\} \\ 
r = 1:&&  \{14, 8, 2, -4, \ldots\} \\
r = 2:&&  \{9, 3, -3, -9, \ldots\} \\
r = 3:&&  \{6, 0 , -6 -12, \ldots\} \\
r = 4:&&  \{1, -5, -11, -17, \ldots\}
\end{eqnarray*}
Since the $c=0$ value for $r=4$ is already less than $r$, there
are only three non-empty rows above the middle (longest) row in
this pinecone.
For the lower portion of the diagram, we obtain the
following values of $L(n,r,c)$:
\begin{eqnarray*}
r = 0:&& \{17, 11, 5, -1, \ldots\} \\ 
r = 1:&& \{16, 10, 4, -4, \ldots\} \\
r = 2:&& \{13, 7, 1, -5, \ldots\} \\
r = 3:&& \{12, 6, 0 , -6, \ldots\} \\
r = 4:&& \{9, 3, -3, -9, \ldots\} \\
r = 5:&& \{8, 2, -4, -10, \ldots\} \\
r = 6:&& \{5, -1, -7, -13, \ldots\}
\end{eqnarray*}
Completing the construction, we arrive at the graph of
Figure~\ref{big-example}.

\begin{figure}[htb]
\begin{center}
\input{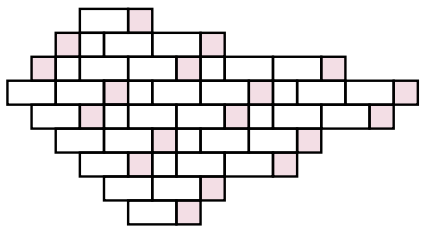}
\caption{The pinecone $P(25;6,2,5,3 )$.}
\label{big-example}
\end{center}
\end{figure}

It is easy to translate this into the format required by the
computer program VAXmaple, written by Greg Kuperberg,
Jim Propp and David Wilson to count perfect matchings of
finite subgraphs of the infinite square grid. In this format,
each vertex present in the graph is represented by a letter.
The choice of letter indicates whether any edges are omitted
when connecting the vertex to its nearest neighbours --- each
vertex having up to four of these. An {\tt X} indicates that no edges
are omitted; an {\tt A} indicates that the edge leading upward from
the vertex is omitted; a {\tt V} indicates the omission of the downward
edge.
(For a more detailed explanation of the software, see 
{\tt http://jamespropp.org/vaxmaple.doc}.) 
The encoding of the pinecone of Figure~\ref{big-example} 
is given in Figure~\ref{vax-example}.

\begin{figure}[htb]
\begin{center}
\begin{verbatim}
                                 XVXX
                                XXAVXVXX
                               XXXVAVAXXVXVXX
                              XVXVAXAVXVAXAVXVXX
                              XAVAXXVAVAXXVAVAXX
                               XAVXVAXAVXVAXAXX
                                XAVAXXVAVAXX
                                 XAVXVAXVXX
                                  XAVAXX
                                   XAXX
\end{verbatim}
\caption{The pinecone $P(25;6,2,5,3)$ as a {\tt VAX} file.}
\label{vax-example}
\end{center}
\end{figure}

Counting the perfect matchings in this pinecone by running the above
input through the VAXmaple program and then through Maple
produces 167,741, as it should, since the $25$th term of the
Gale-Robinson sequence constructed from $(6,2,5,3)$ is 167,741.

\section{The Gale-Robinson bivariate polynomials}\label{sec:bivariate}
As stated in Corollary~\ref{GR-coro}, Theorem~\ref{thm:GR-pinecones}
implies that the three-term Gale-Robinson sequences consist of
integers. In this section, we refine this result as follows.
\begin{Theorem}
\label{thm:GR-refined}
Let $i, j, k, \ell$ and $m$ be positive integers such that
$i+j=k+\ell=m$.
Let $w$ and $z$ be two indeterminates, and define a sequence
$p(n)\equiv p(n;w,z)$ by $p(n)=1$
for $n<m$ and for $n \ge m$,
$$
p(n)p(n-m) =w\, p(n-i)p(n-j)+z\, p(n-k)p(n-\ell).
$$
Then $p(n)$ is a polynomial in $w$ and $z$ with nonnegative integer
coefficients. 
\end{Theorem}
The proof goes as follows: we have already seen that $p(n;1,1)$ counts
perfect matchings of the pinecone $P\equiv P(n;i,j,k,\ell)$ constructed in
Section~\ref{sec:pinecones-GR}. We will prove that $p(n;u^2,v^2)$ counts these
matchings according to two parameters.
More precisely, we begin by giving in
Section~\ref{sec:condensation-refined} a
condensation theorem that computes inductively the
matching polynomial (rather than the matching number) of closed
pinecones.  We observe that this theorem takes a simpler form when
applied to \emm interleaved, pinecones (a class of pinecones that
contains all Gale-Robinson pinecones).
In Section~\ref{sec:active}, we define the \emm special, horizontal edges of a
pinecone. We then define the \emm partial, matching
polynomial of a pinecone $P$ as the matching polynomial $M(P)$ in which 
the weights of non-special edges are set to 1.
We specialize the condensation theorem of
Section~\ref{sec:condensation-refined} to a
condensation theorem for the partial matching polynomial of
interleaved pinecones. Its application to the Gale-Robinson
pinecones $P(n;i,j,k,\ell)$ implies that the polynomial $q(n)\equiv q(n;u,v)$ 
that counts perfect matchings of $P(n)$ according to the number of
vertical edges (the exponent of $v$) and horizontal special edges 
(the exponent of $u$) 
satisfies $q(n)=1$ for $n<m$ and 
$$
q(n)q(n-m) =u^2 q(n-i)q(n-j)+v^2 q(n-k)q(n-\ell)
$$
for $n \ge m$. This shows that $q(n;u,v)=p(n;u^2,v^2)$ 
and implies Theorem~\ref{thm:GR-refined}.

\subsection{A condensation theorem for the matching polynomial}
\label{sec:condensation-refined}
Let us go back to the condensation theorem for closed pinecones
(Theorem~\ref{thm:condensation}).  
We now state and prove a stronger result 
dealing with the matching polynomial rather than the matching number.
Let $P$ be a closed pinecone and 
$A$ the smallest diamond graph that contains it, 
with $G$ defined as in the beginning of
Section~\ref{section:condensation-pinecones}
and with $n$, $s$, $e$, $w$ as in 
Theorem~\ref{thm:kuo} and Figure~\ref{fig:aztec}.
Since $P$ is the core of $G$, the matching polynomial $M(G)$ equals
$M(P) M(G\setminus P)$. Similar
results hold for the sub-pinecones $P^C$, $P^W$, $P^E$, $P^N$ and $P^S$.
Corollary~\ref{coro:G} gives:
\begin{multline}
\label{condensation-pol}
M(P) M(G\setminus P)M(P^C) M(G_C\setminus P^C)
=ns M(P^W) M(G_W\setminus P^W)M(P^E) M(G_E\setminus P^E)\\
+ ewM(P^N) M(G_N\setminus P^N)M(P^S) M(G_S\setminus P^S).
\end{multline}
Since $P$ is the core of $G$, the graph $G\setminus P$ has a unique
perfect matching, which is formed of horizontal edges only. Hence
$M(G\setminus P)$ is a monomial. The same
holds for the other graph differences occurring in~\eqref{condensation-pol}. 
We can thus rewrite this identity as
$$
M(P)M(P^C)=
\alpha M(P^W)M(P^E) + \beta M(P^N)M(P^S)
$$
for some \emm Laurent, monomials $\alpha$ and $\beta$ (indeed, negative
exponents may arise from the division by $M(G\setminus P) M(G_C\setminus P^C)$).
 Our objective in this subsection is to prove that these monomials
 only involve nonnegative 
exponents (so that they are ordinary monomials), and to describe them
in a more concise way. 

We introduce the following definition, illustrated in
Figure~\ref{fig:left}.
\begin{Definition}
Let $P$ be a closed pinecone. 
A horizontal edge is a \emm left, edge 
if it is the leftmost horizontal edge 
in the horizontal segment of $P$ that contains it.  

A horizontal edge with leftmost vertex $(i,j)$ is \emm even,
if $i+j$ is even, \emm odd, otherwise.
\end{Definition}

\begin{figure}[htb]
\begin{center}
\input{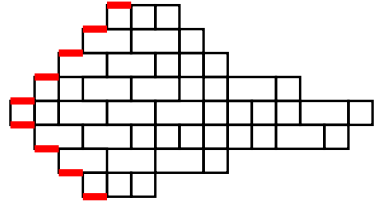}
\caption{The left edges of a closed pinecone.} 
\label{fig:left}
\end{center}
\end{figure}

For a (standard) pinecone $P$, we denote by $P_{\ge}$ (resp.~$P_>$)
the pinecone formed of rows $0,1,2,\ldots$ of $P$ (resp.~rows $1,2,\ldots$).
We use similar notations for the bottom part of $P$. 
These definitions are extended to
transplanted pinecones in a natural way: for instance, 
$(P^N)_\ge$, which we simply denote $P^N_\ge$,
consists of rows $1,2,\ldots$ of $P$. 

\begin{figure}[htb]
\begin{center}
\scalebox{0.85}{\input{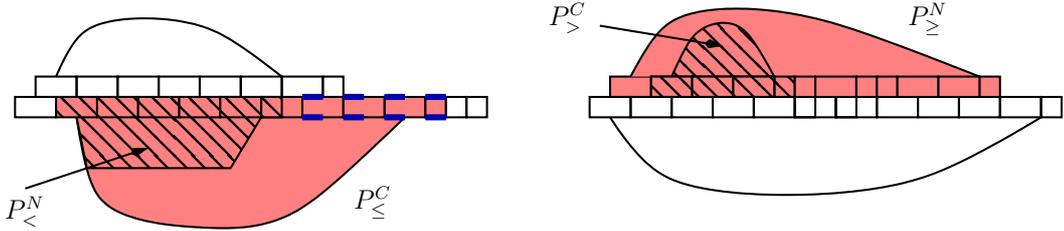}}
\caption{The inclusion properties $P^N_< \subset P^C_{\le}$ and
$P^C_{>} \subset P^N_{\ge}$ (the small subpinecones are dashed). 
The leftmost figure also shows some
edges of the horizontal matching of $P^C_\le \setminus P^N_<$.}
\label{fig:inclusion}
\end{center}
\end{figure}

 Observe that, for any pinecone $P$,
 \beq
\label{inclusions-N}
P^N_< \subset P^C_{\le} \quad \hbox{ while }\quad 
 P^C_{>} \subset P^N_{\ge} .
\eeq
Both properties are illustrated in Figure~\ref{fig:inclusion}.
Consequently, the graph difference 
$P^C\setminus P^N$ is formed of edges  that lie in
$P_\le^C$, and $P^C\setminus P^N\subset P^C_{\le}\setminus P^N_<$.
Let us describe more precisely the horizontal edges of these two graph
differences. For $j \le 1$, if there are any horizontal  edges of 
$P^C_{\le}\setminus P^N_<$ lying at ordinate $j$, then the number
 of them is odd, say $2k_j+1$, and these edges are the $2k_j+1$ rightmost
 horizontal edges of $P^C$ found at ordinate $j$
 (Figure~\ref{fig:inclusion}, left). 
If $j\le 0$, all these edges belong to
 $P^C\setminus P^N$. However, for $j=1$, only a subset of these edges,
 of even cardinality, belong to $P^C\setminus P^N$. (In the example of
 Figure~\ref{fig:inclusion}, the two leftmost thick edges shown at ordinate
 1 do not belong to $P^C\setminus P^N$.)
The   graph $P^C_{\le}\setminus P^N_<$
thus has a unique horizontal 
 matching, which has $k_j+1$ edges at ordinate $j\le 1$. 
 We denote by 
$H^-(P^C\setminus P^N)$ the product of the edges of
 this matching having ordinate $\le 0$. 
The fact that the horizontal edges of $P^C\setminus P^N$ found at
ordinate $j$ coincide with those of $P^C_{\le}\setminus P^N_<$ allows
us to use the notation $H^-(P^C\setminus P^N)$ rather than something
like
$H^-(P^C_\le \setminus P^N_<)$ which would have been heavier.

Symmetrically,
\beq\label{inclusions-S}
P^S_> \subset P^C_{\ge} \quad \hbox{ while }\quad 
 P^C_{<} \subset P^S_{\le} ,
\eeq
 so that the graph $P^C\setminus P^S$
 lies in $P_\ge ^C$. We denote by 
 $H^+(P^C\setminus P^S)$ the product of the edge-weights of the horizontal
 matching of $P^C_\ge \setminus P^S_>$  lying at a positive ordinate.

\begin{figure}[htb]
\begin{center}
\hskip -10mm
\scalebox{1}{\input{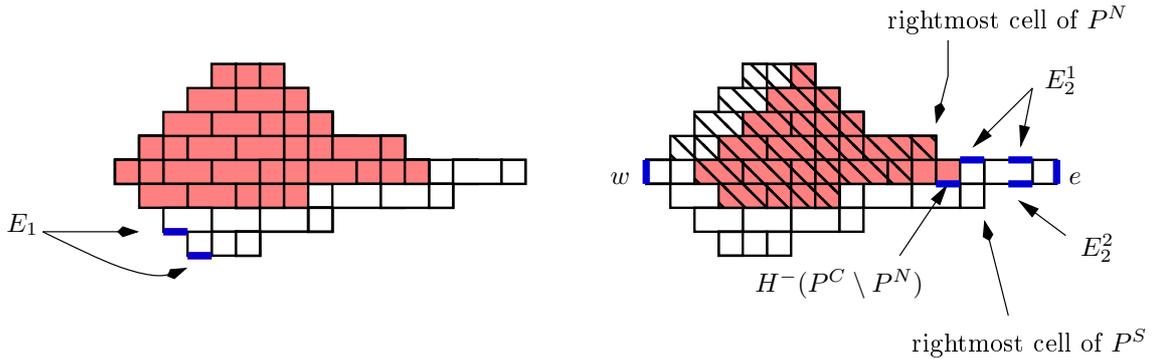}}
\caption{Left: The edges occurring in the first term of the
refined condensation theorem, with the pinecone $P^W$ shown. 
Right: The edges occurring in the second term.  
Here $H^+(P^C\setminus P^S)=1$. The two distinguished
pinecones are $P^C$ and $P^N$ (dashed).}
\label{fig:refined}
\end{center}
\end{figure}

We can now state a condensation theorem for the matching polynomial of
closed pinecones. See Figure~\ref{fig:refined} for an illustration. 

\begin{Theorem}[{\bf The matching polynomial of closed pinecones}]
                  \label{thm:condensation-complete}
The matching polynomial of a closed pinecone $P$ is related to the matching
polynomial of its 
 sub-pinecones by
\begin{multline*}
M(P) M\left(P^C\right)= \left(\prod_{a\in E_1}a\right)
M\left(P^W\right) M\left(P^E\right) \\
+\left(\prod_{a\in E_2}a\right) 
H^-(P^C\setminus P^N)H^+(P^C\setminus P^S)
M\left(P^N\right) M\left(P^S\right),
\end{multline*}
where 
\begin{itemize}
\item $E_1$ is the set of left edges of $P$ not belonging to $P^W$,
\item $E_2$ is the union of three edge-sets $E_2^{i}$, for $0\le i \le 2$:
  \begin{itemize}
  \item $E_2^{0}=\{e,w\}$ contains the eastmost and westmost
    vertical edges of $P$,
\item $E_2^{1}$ contains the even edges at ordinate $1$ not belonging to
  $P^N$,
\item $E_2^{2}$ contains the odd edges at ordinate $0$ 
not belonging to $P^S$.
  \end{itemize}
\end{itemize}
\end{Theorem}

\noindent
\begin{proof}
In this proof, we adopt the following notation: for each edge set $E$,
we also denote by $E$ the product of the edges of the set. For a graph
$G$ having a unique horizontal (perfect) matching, we denote 
this matching by $H(G)$.

Let us return to~\eqref{condensation-pol}. Recall that $G\setminus P$
has a unique matching, consisting of horizontal edges only.
Denoting by $A$ the
smallest diamond graph containing $P$, we observe that $M(G\setminus P)=
H(A)/H(P)$ (see Figure~\ref{fig:completion-pinecone}). 
Similar identities hold for the other pinecones occurring
in~\eqref{condensation-pol}.  For instance, $M(G_W\setminus P^W)=
H(A_W)/H(P^W)$. This allows us to rewrite~\eqref{condensation-pol} as 
\begin{multline*}
M(P)M(P^C)= ns\ \frac{H(A_W)H(A_E)}{H(A)H(A_C)} \ 
\frac{H(P)H(P^C)}{H(P^W)H(P^E)}\ M(P^W)M(P^E)\\
+
ew\ \frac{H(A_N)H(A_S)}{H(A)H(A_C)}\ \frac{H(P)H(P^C)}{H(P^N)H(P^S)}
\ M(P^N)M(P^S).
\end{multline*}
Let us begin with the two factors involving $A$ and its
subgraphs. It is easy to see, with the help of
Figure~\ref{fig:sub-diamonds}, that 
$$
 ns\ \frac{H(A_W)H(A_E)}{H(A)H(A_C)}=1.
$$
The second factor involving $A$, namely $ew{H(A_N)H(A_S)}/{(H(A)H(A_C))}$,
is a multiple of $e$ and $w$ (all the other edges are horizontal) and
thus cannot be equal to 1.   
Denoting by $L_1^{(e)}$  
the graph formed by the even horizontal edges lying
at ordinate 1, and introducing similar notations 
$L_1^{(o)}$, $L_0^{(e)}$ and $L_0^{(o)}$, one finds
\beq\label{HA}
 \frac{H(A_N)H(A_S)}{H(A)H(A_C)}= \frac{L_1^{(e)} L_0^{(o)}}
{L_1^{(o)} L_0^{(e)}}.
\eeq
It remains to describe the two factors that involve $P$ and its subgraphs. 
For the first one, we note that $H(P)/H(P^E)$ is simply the product of the
left edges of $P$. Similarly, as $P^C= (P^W)^E$, the ratio
$H(P^W)/H(P^C)$ is the product of the left edges of $P^W$. This gives
the following expression for the first factor:
$$
\frac{H(P)H(P^C)}{H(P^W)H(P^E)}= \prod_{a\in E_1} a,
$$
with $E_1$ defined as in the theorem.

To express the second factor involving $P$, let us first separate in
$H(P)$ the edges that lie at ordinate $j=0$, $j=1$, $j 
>1$, $j<0$. This gives
$$
H(P)= L_0^{(e)} \cdot L_1^{(o)}
\cdot \frac{H(P^N_ \ge)}{L_1^{(e)} \cap P^N} \cdot
\frac{H(P^S_\le)}{ L_0^{(o)} \cap P^S }.
$$
For the other 3 pinecones that are involved in this factor, we write:
$$
H(P^C)= \frac {H(P^C_\ge) H(P^C_\le)}{(L_0^{(e)} \cap P^C)(L_1^{(o)}
  \cap P^C)}, 
\quad H(P^N)= \frac{H(P^N_\ge) H(P^N_<)}{L_1^{(o)} \cap P^N_<},
\quad 
H(P^S)= \frac{H(P^S_\le) H(P^S_>)}{L_0^{(e)} \cap P^S_>}.
$$
The division by ${(L_0^{(e)} \cap P^C)(L_1^{(o)}  \cap P^C)}$ in the
first identity comes from the fact that $ H(P^C_\ge)$ and $
H(P^C_\le)$ have edges in common at ordinates 0 and 1. The other
divisions are justified in a similar way.
These identities, together with~\eqref{HA}, give:
$$
  \frac{H(A_N)H(A_S)}{H(A)H(A_C)}\ \frac{H(P)H(P^C)}{H(P^N)H(P^S)}
=
(L_1^{(e)} \setminus P^N) (L_0^{(o)} \setminus P^S)\ 
 \frac{( L_1^{(o)} \cap P^N_< ) H(P^C_\le)}{( L_1^{(o)} \cap P^C )
   H(P^N_<)}\ 
 \frac{( L_0^{(e)} \cap P^S_> ) H(P^C_\ge)}{( L_0^{(e)} \cap P^C )
   H(P^S_>)}.
$$
The ratio $ H(P^C_\le)/ (L_1^{(o)} \cap P^C) $ is the product of the
edges found at non-positive ordinates in the horizontal matching of $P^C_\le$.
Similarly, the ratio $H(P^N_<)/( L_1^{(o)} \cap P^N_< )$ is the product of the
edges found at non-positive ordinates in the horizontal matching of $P^N_<$.
But $P^N_<\subset P^C_\le$ (see~\eqref{inclusions-N} and its
accompanying Figure~\ref{fig:inclusion}), so the quotient of the two
ratios is $H^-(P^C\setminus P^N)$, the product of the edges found at
non-positive ordinates in the horizontal matching of $P^C_\le \setminus P^N_<$.
The remaining quotient involving $P^S_>$ is, symmetrically, equal to $
H^+(P^C\setminus P^S)$. This yields the result stated in the theorem.
\end{proof}

The refined condensation theorem specializes nicely to \emm
interleaved, pinecones.

\begin{Definition}[{\bf Interleaved pinecones}]
A closed pinecone is \emm interleaved, if, between two black squares in
row $r$, one finds a black square in row $r+1$ and a black square in
row $r-1$. 
\end{Definition}
This implies that, between two \emm consecutive, black
squares in row $r$, there is exactly one black square in row $r+1$, 
and one in row $r-1$. For instance, the pinecone to the right of
Figure~\ref{fig:inactive} is interleaved. Going back to
Theorem~\ref{thm:condensation-complete}, 
it is easy to see that for an interleaved pinecone, 
the graphs 
$P^C_\le\setminus P^N_<$ and $P^C_\ge\setminus P ^S_>$ 
are empty, so that $H^-(P^C\setminus P^N)= H^+(P^C\setminus P ^S)=1$.

\begin{coro}[{\bf The matching polynomial of interleaved pinecones}]
\label{coro:condensation-interleaved}
The matching polynomial of an interleaved pinecone $P$ is related to
the matching polynomial of its 
 closed sub-pinecones by
$$
M(P) M(P^C)= \left(\prod_{a\in E_1}a\right)
M(P^W) M(P^E)
+\left(\prod_{a\in E_2}a\right) 
M(P^N) M(P^S),
$$
where the sets $E_1$ and $E_2$ are described in
Theorem~\ref{thm:condensation-complete}. Moreover, the five
sub-pinecones of $P$ are also interleaved.
\end{coro}
The last statement follows from the fact that each of the five
sub-pinecones can be defined as the largest closed pinecone containing two
prescribed vertical edges. 

\subsection{Special edges}\label{sec:active}

We will now simplify further the expression of
Corollary~\ref{coro:condensation-interleaved}, by assigning
weight 1 to certain horizontal edges, called \emm ordinary,. If $P$ is
interleaved, the set of black squares of $P^W$ is obtained by deleting the
  rightmost black square in each row of $P$. Consequently, the rows
  that disappear when constructing $P^W$ from $P$ are those that
  contain only one black square. These are the rows that contain a
  left edge contributing to the
  set $E_1$. Moreover, the top and bottom rows of $P$ 
  contain exactly one black square, otherwise $P$ would not be
  interleaved. Hence $E_1$ has cardinality at least 2. We are going to
  assign weight 1 to all the edges of $E_1$ that lie neither on the
  top segment of $P$ nor on its bottom segment. 
Similarly, we will assign weight 1 to the edges of $E^1_2$ and 
$E_2^2$, so that the product of the edge-weights in $E_2$ will reduce to
$ew$.
 As we want to apply the condensation theorem
iteratively, this forces us to set to 1 the weights of \emm other, horizontal
edges, occurring for instance
in the sets $E^1_2$ and $E_2^2$ associated to the
five sub-pinecones of $P$. Iterating this procedure, 
we arrive at the
following definition of \emm ordinary, horizontal edges --- 
those that will have
weight 1. This definition is illustrated in
Figure~\ref{fig:inactive}. Note that it does not assume that the pinecone is
interleaved. 

\begin{Definition}\label{def:inactive}
An even horizontal edge $a$, lying 
at ordinate $r$ in a pinecone
(that is, between rows $r-1$ and $r$),  
is \emm ordinary, 
if the closest black square found in rows $r-1$ and $r$ weakly to the
right of $a$ is in row $r-1$. Otherwise, $a$ is said to be \emm special,. 
In particular, if an even edge $a$ lies in the bottom segment of $P$, 
it is special. 

Symmetrically, an odd horizontal edge $a$, lying at ordinate $r$, is
ordinary if the closest black square found in rows $r-1$ and $r$ weakly to
the right of $a$ is in row $r$. Otherwise, $a$ is said to be special. 
In particular, if an odd edge $a$ lies in the top segment of $P$, 
it is special. 

\end{Definition}

\begin{figure}[htb]
\begin{center}
\input{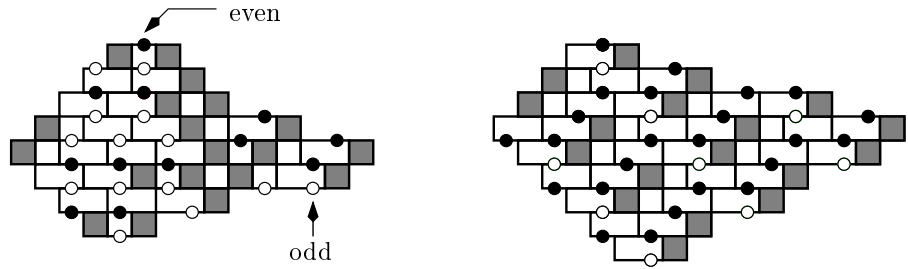}
\caption{The ordinary edges of a pinecone. The even ones are in black,
  the odd ones in white. The pinecone to the right is interleaved.} 
\label{fig:inactive}
\end{center}
\end{figure}

It is easy to check that in an interleaved pinecone,
the edges of $E_2^1$ and $E_2^2$ are ordinary. The following lemma
tells which edges of $E_1$ are special.
\begin{Lemma}\label{lemma:left-interleaved}
  Let $P$ be an interleaved closed pinecone. 
 There are exactly two left
  edges of $P$ that do not belong to $P^W$ and are special. One of them
  is even, and is the lowest left edge of $P$. The other is odd, and
  is the highest left edge of $P$.
\end{Lemma}
\begin{proof} As noted at the beginning of this subsection,
the left edges of $P$ that do not belong to $P^W$ are those that
belong to rows containing exactly one black square. Take an even edge
of this type. It belongs to the bottom portion of
$P$. Figure~\ref{fig:left-active} shows that it is always ordinary,
unless it lies on the bottommost horizontal segment of $P$. The proof
is similar for odd left edges.
\end{proof}

\begin{figure}[htb]
\begin{center}
\input{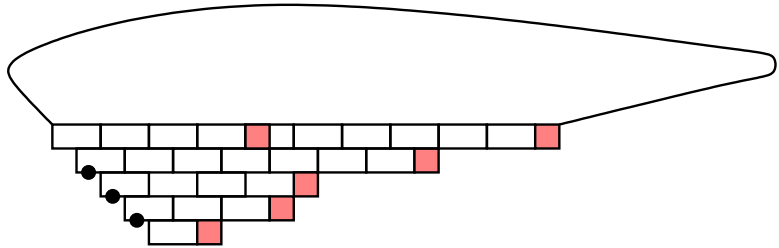}
\end{center}
\caption{The even ordinary left edges of $P\setminus P^W$.}
\label{fig:left-active}
\end{figure}

\begin{Lemma}\label{lem:stable}
Let $P$ be a closed pinecone, and $Q$ one of the five sub-pinecones
$P^C$, $P^W$, $P^E$, $P^N$, $P^S$.  The ordinary edges of $Q$ are exactly the
ordinary edges of $P$ belonging to $Q$.
\end{Lemma}
\begin{proof} Let $a$ be an even ordinary edge of $P$, lying at ordinate
$r$. Let $c$ be the first black square found in row $r-1$ weakly to the
right of $a$. By definition of ordinary edges, there is no black
square in row $r$ between $a$ and $c$. Assume $a$ belongs to $Q$ and
is \emm not,  ordinary in $Q$. Since we do not add squares when going from $P$
to $Q$, this means that $c$ does not belong to $Q$. Then there is no
black square in row $r-1$ to the right of $a$ in $Q$.  However, since
$a$ belongs 
to $Q$, there must be a black square $c'$ to the right of $a$ in row
$r$ of $Q$. This square $c'$ is also in $P$, and to the right of $c$.
But $Q$ is defined as the \emm largest, closed subpinecone of $P$ having
certain prescribed rightmost and leftmost edges, so that if it
contains $a$ and $c'$, it has to contain $c$ as well. We have thus
reached a contradiction, and $a$ is ordinary in $Q$.
\begin{center}
\input{Figures/inactive-active.pstex_t}
\end{center}

Conversely, assume $a$ is special in $P$, but ordinary in $Q$. The
latter property implies that there is a
black square $c$ in row $r-1$ of $Q$ to the right of
$a$. Of course, $c$ also belongs to $P$. Since $a$ is special in $P$,
there is a black square $c'$ in row $r$ of $P$ lying between $a$
and 
$c$. As $Q$ is the largest pinecone containing two prescribed edges,
and contains $a$ and $c$, the square $c'$ must be in $Q$ as well,
contradicting the assumption that $a$ is ordinary in $Q$.
\begin{center}
\input{Figures/active-inactive.pstex_t}
\end{center}
Of course, the proof is completely similar for odd special edges.
\end{proof}

\subsection{The partial matching polynomial}
For any pinecone $P$, define its \emm partial, matching polynomial
$\tilde M (P)$ to be
the value of $M(P)$ when the weights of all ordinary edges are set to
1.  We emphasize that this polynomial counts \emm perfect,
matchings (all vertices of $P$  belong to an edge in the
matching), but some of the edges have weight 1.
Assume $P$ is interleaved, and apply
Corollary~\ref{coro:condensation-interleaved}. As 
observed after Definition~\ref{def:inactive}, all the
edges of $E_2^1$ and $E_2^2$ are ordinary, so that they have weight
1. This means that the second monomial occurring in the condensation
formula is simply $ew$. Moreover, 
the special edges of $E_1$ are described in
Lemma~\ref{lemma:left-interleaved}. 
This, combined with Lemma~\ref{lem:stable}, implies the following
corollary.
\begin{coro}
\label{thm:condensation-interleaved}
The partial matching polynomial of an interleaved closed pinecone $P$
 is related to the partial matching polynomials of its 
 sub-pinecones by 
$$
\tilde M(P) \tilde M(P^C)= a a'
\tilde M(P^W) \tilde M(P^E) +ew
\tilde M(P^N) \tilde M(P^S),
$$ 
where $a$ and $a'$ are the highest and lowest left edges of $P$.
\end{coro}

Since the Gale-Robinson pinecones constructed in Section~\ref{sec:pinecones-GR}
are interleaved, we have obtained a combinatorial 
interpretation of the Gale-Robinson polynomials.

\begin{Theorem}
\label{thm:GR-pinecones-refined}
Let $P(n)\equiv P(n;i,j,k,\ell)$ be the sequence of pinecones
associated with the parameters $(i,j,k,\ell)$. 
 Let $q(n)\equiv (n;u,v)$ be the polynomial in $u$ and $v$
that counts the perfect matchings of $P(n)$ according to the number of
vertical edges (the exponent of $v$) 
and horizontal special edges (the exponent of $u$). 
Then $q(n)=1$ for $n<m$ and for $n \ge m$,
$$
q(n)q(n-m) =u^2 q(n-i)q(n-j)+v^2 q(n-k)q(n-\ell).
$$
\end{Theorem}
This proves Theorem~\ref{thm:GR-refined}, as the recurrence shows that
$q(n;u,v)=p(n;u^2,v^2)$.

\section{Perspectives}

\subsection{Variations and extensions}
There is a good deal of overlap between this article
and the paper by David Speyer on the general octahedron
recurrence, of which the Gale-Robinson recurrence is
a very special case~\cite{Sp}.  Speyer's method
allows him to construct, for each $(i,j,k,\ell)$ with
$i+j=k+\ell$, a sequence of graphs having the same number
of perfect matchings as the pinecones we construct.
We believe that our graphs are the same as the ones that 
are given by Speyer's procedure,
but we have not proved that this holds in general.

One undesirable feature of our description of Gale-Robinson pinecones
is that it breaks some of the symmetries
between the parameters $i$, $j$, $k$, and $\ell$.
Clearly, exchanging $k$ and $\ell$ reflects the pinecone
across a horizontal line. 
But the convention $j=\min\{i,j,k,\ell\}$ implies that $i$ and $j$ do
not play symmetric roles, nor the pairs $\{i,j\}$ and $\{k,\ell\}$.
This explains why the description of the
bivariate polynomials $q(n;u,v)$ of
Theorem~\ref{thm:GR-pinecones-refined} is not symmetric in $u$ and $v$.
Perhaps some of this asymmetry is unavoidable, but
it would be good to find a more symmetrical definition
or else achieve some insight into 
why the asymmetry cannot be avoided.  

Indeed, part of the point of view that led to both
this article and Speyer's is that the truly fundamental
objects of study are functions that map a three-dimensional
lattice to some ring and that obey the octahedron relation
$$f({\bf x} + {\bf i}) f({\bf x} - {\bf i})
+ f({\bf x} + {\bf j}) f({\bf x} - {\bf j})
+ f({\bf x} + {\bf k}) f({\bf x} - {\bf k}) = 0$$
(where ${\bf x}$ is an arbitrary vector in the lattice
and ${\bf i}$, ${\bf j}$, ${\bf k}$ are fixed
generators of the lattice)
and more general versions of the relation that
include coefficients of various kinds.
There is no intrinsic ``arrow of time'' here
(as there is when one thinks of running a recurrence relation
forward from some set of initial conditions),
but some sets of initial conditions are sufficiently large
that they allow one to reconstruct the entirety of $f$,
and some of these subsets of the lattice
can be viewed as ``space-like'',
so that one can think of the reconstruction 
of successive slices of the lattice as a kind of propagation.
In the fully symmetrical version,
there is no reason to privilege one direction over its reverse,
or one axis over another.

In contrast, when one descends from this level
to the more concrete world of graphs and perfect matchings,
the symmetry appears to be broken.
A full theory of the octahedron recurrence
would incorporate graph-theoretic analogues
of all the symmetries of the three-dimensional lattice;
such an understanding is currently lacking.
Just as Ehrhart theory for enumeration of lattice-points in polytopes
can best be understood in a context
that includes inside-out polytopes~\cite{BZ},
the theory of Aztec diamonds, crosses-and-wrenches, and pinecones
requires notions of geometric graphs in which
combinatorial parameters that are ordinarily required to be positive
can take on negative values as well.
(E.g., one needs a theory in which the notion of an Aztec diamond 
of order 4 and an Aztec diamond of order $-4$ enter on an equal
footing, and the latter graph turns out to be
essentially the same things as an Aztec diamond graph of order 3.)
As a hint of what such a theory might look like,
the interested reader should look at~\cite{Pr1} and~\cite{ABBP}.

\medskip

The bivariate polynomials $p(n;w;z)$
studied in Section~\ref{sec:bivariate}
generalize Gale-Robinson numbers.  A different extension of these
numbers comes from replacing the initial conditions (a string of $m$ 1's)
by generic initial conditions (indeterminates $x_0$ through $x_{m-1}$).
Here again, Fomin and Zelevinsky proved algebraically,
and Speyer proved combinatorially,
that the rational functions one obtains
are Laurent polynomials in $x_0,\dots,x_{m-1}$.
Speyer's work shows that these variables,
in contrast to the formal coefficients $w$ and $z$ mentioned above,
are most naturally viewed as being associated
with the faces of a graph, rather than its edges.
So there should be a way to associate these $m$ variables
with the faces of our pinecones
and use them to assign weights to the perfect matchings
so that the weight of each perfect matching of a pinecone
is a Laurent monomial in $x_0,\dots,x_{m-1}$.
Indeed, there should be an extension of
Theorem~\ref{thm:GR-pinecones-refined}
that describes the Laurent polynomials that arise
from setting $a(n)=x_n$ for $0 \leq n \leq m-1$
and $a(n) = (wa(n-i)a(n-j)+za(n-k)a(n-\ell))/a(n-m)$ for $n \geq m$,
and in particular identifies each Laurent monomial in $a(n)$
as the weight of a perfect matching of $P(n;i,j,k,\ell)$.

Most of the work of this article was done in 2005 and 2006,
as the study of cluster algebras was beginning its (still continuing)
outward explosion, so there are now other approaches to
proving positivity results that have some overlap with
the approach taken here.  In particular, it is possible
that pinecones graphs can also be viewed as 
Aztec diamond graphs with defects, in the manner of~\cite{DFZ}.

\subsection{Random pinecone matchings}
A rather different direction that might be studied
is the ``typical'' behavior of perfect matchings of large pinecones.
Figure~\ref{fig:random}~shows two 
tilings associated with matchings of Somos-4 pinecones. 
(Here we make use of the standard duality between 
a tiling of a polyomino by dominos 
and a perfect matching of the dual graph of the polyomino,
in which vertices correspond to cells of the polyomino
and edges correspond to pairs of adjoining cells,
i.e.\ legal positions of a domino in a tiling.)
The first one corresponds to $n=26$
(that is, to a perfect matching of the graph $P(26;3,1,2,2)$), 
the second one to $n=50$.
Both were chosen \emm uniformly at random, from the set of all
perfect matchings of that graph.
These examples were produced using Propp and Wilson's papers on ``exact
sampling''~\cite{PW1,PW2} which show how the method of ``coupling from
the past'' permits one to generate random perfect matchings 
of bipartite planar graphs.
Indeed, this algorithm was incorporated into a program called {\tt vaxrandom}
that accepts a VAX-file as input and produces
a perfect matching of the associated graph as output, or rather,
the dual picture of a domino tiling of a region.
The source code for the program is contained in the files
{\tt http://jamespropp.org/tiling/sources/vaxrandom.c}
and
{\tt http://jamespropp.org/tiling/sources/allocate.h},
and information on the program's use can be found at
{\tt http://jamespropp.org/tiling/doc/vaxrandom.html}.)

\begin{figure}[hbt]
\begin{center}
 \vskip -10mm  \includegraphics[angle=270, scale =0.3]{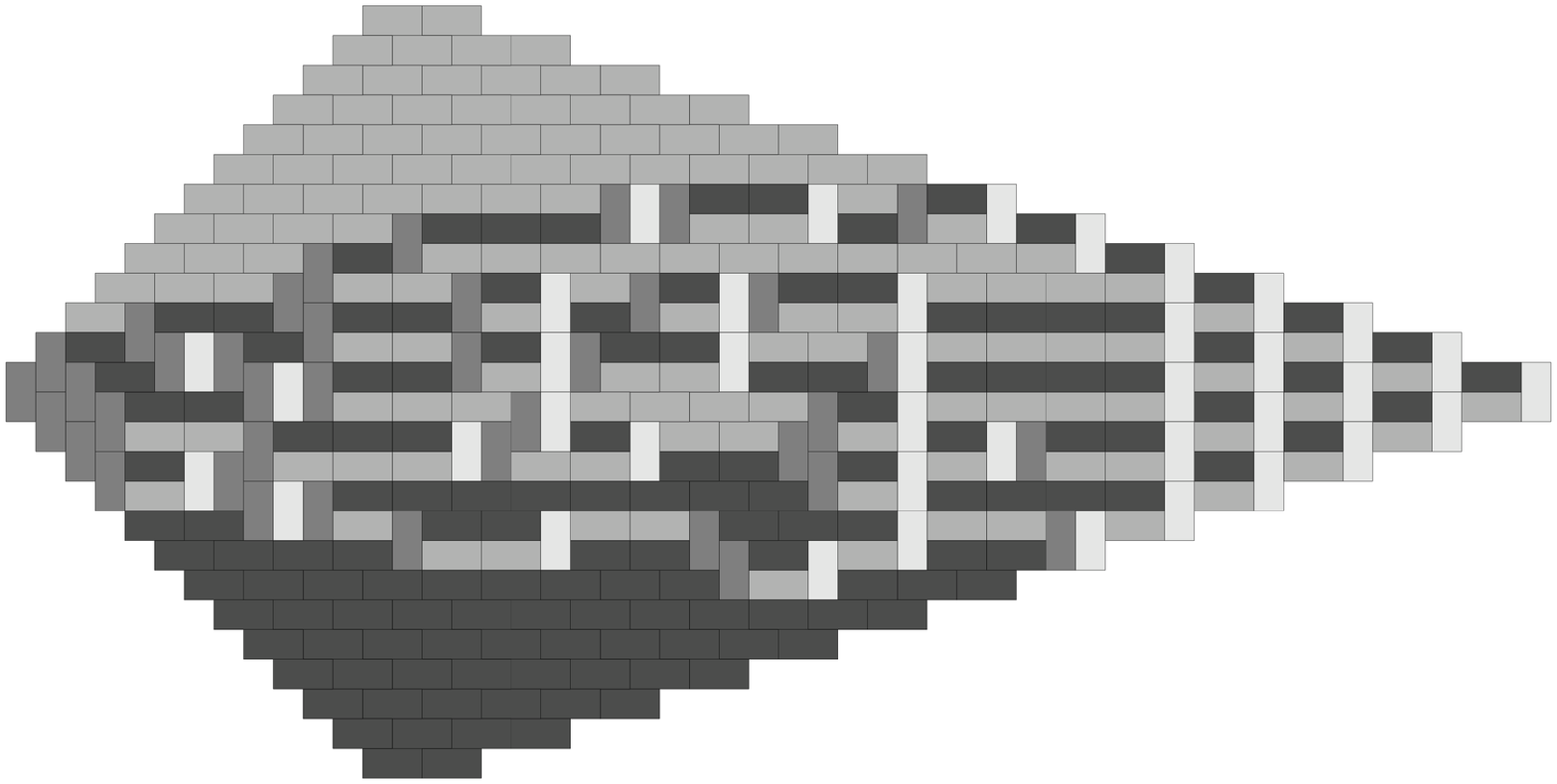}  \\ 
\includegraphics[angle=270, scale =0.3]{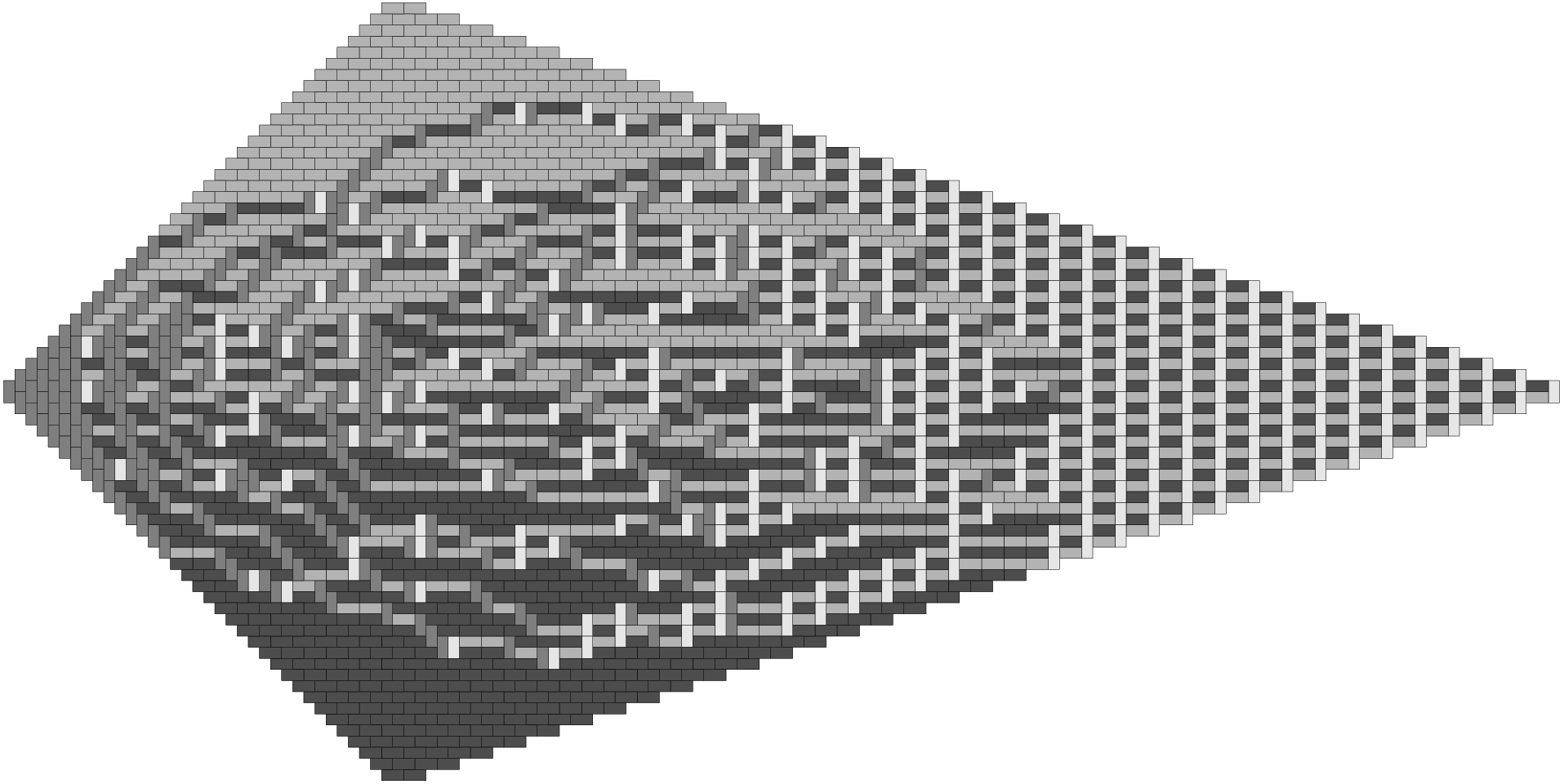}   
\end{center}
\vskip -10mm \caption{The domino tilings associated with random
perfect matchings of the pinecone $P(n; 1, 3, 2, 2)$, for $n=26$ and
then $n=70$.}
\label{fig:random}
\end{figure}

The reader will quickly notice that in both of these random tilings,
the randomness is not spatially distributed in a uniform manner.
Near the boundary, there is a good deal of
order, with tiles lined up the same way as their neighbors;
only in the interior does one find random-looking behavior.

This phenomenon is not specific to pinecones, but has been observed
for a wide variety of two-dimensional tiling models over the past decade,
from~\cite{CEP} and~\cite{CLP} to~\cite{KOS}.
The most-studied case is the Aztec diamond graph ($(i,j,k,\ell)=(1,1,1,1)$,
in our notation); in this case, it has been shown that in a suitable
asymptotic sense there is a sharp boundary between the part of the
tiling that is random and the part that is orderly, and that this
boundary is (asymptotically) a perfect circle.  A similar sort of
domain-boundary is visible in 
Figure~\ref{fig:random};
assuming that the theory for pinecones is analogous
to the theory for Aztec diamond graphs,
it would be interesting to know the asymptotic shape
of the domain-boundary for $(i,j,k,\ell)$-pinecones
as $n \rightarrow \infty$.

One interesting feature of Gale-Robinson pinecones is that
we can write the definition in a way that makes sense even when 
the parameters $(i,j,k,\ell)$ cease to be integers.
Formula (\ref{UL-def}) can be rewritten as
\beq
\begin{array}{lll}
U(t,r,c) &= &2c + r -3 - 2 \left\lfloor{\displaystyle 
\mu c + \kappa r + \iota - t}\right\rfloor , \\
\\
L(t,r,c) &=& 2c + r -3 - 2 \left\lfloor{\displaystyle 
\mu c + \lambda r + \iota - t}\right\rfloor.
\end{array}
\eeq
where $\iota = i/j$, $\kappa = k/j$, $\lambda = \ell/j$, 
$\mu = m/j = \iota + 1 = \kappa + \lambda$, and $t = (n+1)/j$.
So there is a sense in which all the pinecones
discussed in this article
are part of a four parameter family,
parametrized by $\iota$, $\kappa$, $\lambda$ and $t$.
Of course, the graphs do not vary continuously in these variables
(being discrete elements in a countable set,
namely the set of all finite graphs,
how could they?),
but this parametrization seems likely to be natural
for some purposes,
e.g., the study of random perfect matchings of pinecones.
(It is to be expected that a coherent limit-law 
with $t \rightarrow \infty$ will prevail for
any fixed choice of $(\iota,\kappa,\lambda)$,
whether or not $\iota$, $\kappa$, and $\lambda$ are rational.)
It should be noted, incidentally,
that if one chooses parameters $(i,j,k,\ell)$
with a greater common divisor $d > 1$,
the sequence of pinecones one gets from our construction
is the same as the sequence of pinecones that one gets
from the parameters $(i/d,j/d,k/d,\ell/d)$,
except that each pinecone in the latter sequence
is repeated $d$ times in the former sequence;
this observation follows easily from
the $\iota,\kappa,\lambda$ formulation
of the definitions of $U(\cdot)$ and $L(\cdot)$.

\subsection{Closed-form expressions}
One feature common to sequences
satisfying three-term or four-term Gale-Robinson recurrences
is that the terms grow at quadratic exponential rate. 
Indeed, it is easy to verify, 
from the discussion of pinecones,
that in the infinite sequence of graphs
associated with any particular three-term Gale-Robinson recurrence,
the $n$th graph has $O(n^2)$ vertices,
with each vertex having degree at most 4.
It follows from this that Gale-Robinson sequences
have at most exponential-quadratic growth;
that is, the $n$th term is bounded above by $C^{n^2}$
for all sufficiently large $C$.
In some cases, an exact formula is possible;
we have already mentioned the ``Aztec diamond case'' $i=j=k=\ell=1$,
and in the case $(i,j,k,\ell,m)=(6,1,4,3,7)$
there is an exact formula for $a(n)$
of the form $2^{e_2(n)} 3^{e_3(n)}$
where the exponents $e_2(n)$ and $e_3(n)$
are given by quadratic polynomials in $n$
whose coefficients are periodic functions of $n$
(we thank Michael Somos for bringing this special case 
of the Gale-Robinson recurrence to our attention,
and we raise the question of whether there are other instances 
of Gale-Robinson sequences being given by simple exact formulas).
However, in general such algebraic formulas do not exist.
Instead, one must be content with formulas that express
the $n$th term in terms of Jacobi theta functions.
This link with the analytic world is what motivated
Michael Somos to introduce the Somos-$k$ sequences to begin with.
E.g., back in 1993, Somos announced (without proof)
that the $n$th term of the Somos-6 sequence is given by
$f(n-2.5,n-2.5)$
where
$$f(x,y) = c_1 c_2^{xy} \sum_{k_2=-\infty}^{\infty}
(-1)^{k_2} \sum_{k_1=-\infty}^{\infty} g(k_1,k_2,x,y),$$
$$g(k_1,k_2,x,y) = c_3^{k_1^2} c_4^{k_2^2} c_5^{k_1 k_2} 
\cos(c_6 k_1 x + c_7 k_2 y),$$
$$c_1 = 0.875782749065950194217251...,$$
$$c_2 = 1.084125925473763343779968...,$$
$$c_3 = 0.114986002186402203509006...,$$
$$c_4 = 0.077115634258697284328024...,$$
$$c_5 = 1.180397390176742642553759...,$$
$$c_6 = 1.508030831265086447098989...,$$ 
$$c_7 = 2.551548771413081602906643...$$
(See {\tt http://jamespropp.org/somos/elliptic}
for a similar but simpler formula for the Somos-4 sequence.)
However, as far as we are aware,
nobody has proposed (or even conjectured) a fully general 
analytic formula for the terms of sequences
satisfying three-term Gale-Robinson recurrences.
A more detailed discussion of the analytic properties
of such sequences can be found in~\cite{Hone},
which also gives some of the history of these sequences.

It is worth mentioning that for the Somos-4 sequence,
there exists a unique constant $c$
such that $s(n)$ (the $n$th term of the sequence)
is on the order of $c^{n^2}$,
but that the behavior of $s(n)/c^{n^2}$ is oscillatory;
see {\tt http://jamespropp.org/somos/elliptic}. 
Let us also mention a
recent paper by Xin~\cite{xin}  where the Somos-4 numbers are
expressed as  determinants of Hankel matrices with integer coefficients.

\subsection{Analogy with the KP hierarchy}
We conclude with some remarks (based on some unpublished
remarks of Andrew Hone) about the analogy between Somos 
sequences and the like and the hierarchy of solutions to 
an integrable PDE like the KdV equation, followed by our
own speculation about a direction for further study that
the analogy might suggest.

The equation $$u_{xxx} + 6 u u_x + u_t = 0,$$
where $u = u(x,t)$ is the function we want to solve for
and subscripts indicate partial differentiation
(e.g., $u_{xxx} = \frac{\partial^3 u}{\partial x^3}$)
is known as the KdV equation,
and has played a crucial role in the modern theory of
partial differential equation,
as part of a large family of equations
with related properties (the ``KP hierarchy'').
If one sets $u = 2 (\partial_x)^2 \log F$
one can rewrite the PDE in the compact form
$$(D_x D_x D_x D_x + D_x D_t) (F \otimes F) = 0$$
where $D_x$ and $D_t$ are the ``Hirota $D$-operators''
acting on tensor-pairs of functions via
$$D_x (f(x,t) \otimes g(x,t)) = 
(\partial_{x_1} - \partial_{x_2}) f(x_1,t) g(x_2,t) |_{x_1 = x_2 = x}$$
and
$$D_t (f(x,t) \otimes g(x,t)) = 
(\partial_{t_1} - \partial_{t_2}) f(x,t_1) g(x,t_2) |_{t_1 = t_2 = t}.$$
(Note that in the literature on KdV, this tensor product is traditionally
written as $f \cdot g$ rather than $f \otimes g$ and is called the
``dot-product'', but it is a tensor product, not an inner product).
More generally, the bilinear method is the trick of rewriting 
PDEs in the form $P(D_x, D_y, \ldots) (F \otimes F) = 0$.
Hirota operators are antisymmetric, 
so we can think of them as actions
on the antisymmetric square of a vector space of functions.
For more on the Hirota method, see e.g.~\cite{Hirota}.

Analogously, if we take $V$ to be the
vector space of real- (or complex-) valued
bilaterally infinite sequences $(\ldots,s_0,\ldots)$, 
we may define, for every pairs of integers $i,j$,
a \emm bilinear shift operator $V \otimes V \rightarrow V$,
sending $(s_n)_{-\infty}^{\infty} \otimes (t_n)_{-\infty}^{\infty}$ 
to $(s_{n+i} t_{n+j})_{-\infty}^{\infty}$
(the sequence whose $n$th term is
$s_{n+i} t_{n+j}$ for all $n \in \zs$).
These operators, graded by $i+j$,
generate a graded ring of bilinear shift-operators,
and the Somos sequences and Gale-Robinson sequences
are special instances of sequences $(s_n)_{-\infty}^{\infty}$
for which the tensor-square 
$(s_n)_{-\infty}^{\infty} \otimes (s_n)_{-\infty}^{\infty}$
lies in the kernel of a particular bilinear operator.
It has been noticed that for a typical Somos or Gale-Robinson sequence,
the tensor-square of the sequence,
in addition to being annihilated by the ``defining'' bilinear operator,
is annihilated by infinitely many others as well.
In fact, there is more than just an analogy at work here:
each GR recurrence can be written in terms of Hirota differential operators
by taking exponentials (see~\cite{SS}).
Hopefully, by combining algebraic, analytic, and combinatorial tools,
future researchers will shed some light on this intriguing phenomenon.

\bibliographystyle{plain}
\bibliography{GR.bib}

\begin{thebibliography}{10}

\bibitem{ABBP}
N.~Anzalone, J.~Baldwin, I.~Bronshtein, and T.~K. Petersen.
\newblock A reciprocity theorem for monomer-dimer coverings.
\newblock In {\em Discrete models for complex systems, {DMCS} '03 ({L}yon)},
  Discrete Math. Theor. Comput. Sci. Proc., AB, pages 179--193 (electronic).
  Assoc. Discrete Math. Theor. Comput. Sci., Nancy, 2003.
\newblock Available on the arXiv at {\tt math.CO/0304359}.

\bibitem{BZ}
M.~Beck and T.~Zaslavsky.
\newblock Inside-out polytopes.
\newblock {\em Adv. Math.}, 205(1):134--162, 2006.

\bibitem{BQ}
A.~T. Benjamin and J.~J. Quinn.
\newblock {\em Proofs that really count: The art of combinatorial proof},
  volume~27 of {\em The Dolciani Mathematical Expositions}.
\newblock Mathematical Association of America, Washington, DC, 2003.

\bibitem{berstel-reutenauer}
J.~Berstel and C.~Reutenauer.
\newblock Another proof of {S}oittola's theorem.
\newblock {\em Theoret. Comput. Sci.}, 393(1-3):196--203, 2008.

\bibitem{BP}
D.~Bressoud and J.~Propp.
\newblock How the alternating sign matrix conjecture was solved.
\newblock {\em Notices Amer. Math. Soc.}, 46(6):637--646, 1999.

\bibitem{CS}
G.~Carroll and D.~Speyer.
\newblock The cube recurrence.
\newblock {\em Electron. J. Combinat.}, 11(1)(article R73), 2004.

\bibitem{CEP}
H.~Cohn, N.~Elkies, and J.~Propp.
\newblock Local statistics for random domino tilings of the {A}ztec diamond.
\newblock {\em Duke Math. J.}, 85(1):117--166, 1996.

\bibitem{CLP}
H.~Cohn, M.~Larsen, and J.~Propp.
\newblock The shape of a typical boxed plane partition.
\newblock {\em New York J. Math.}, 4:137--165 (electronic), 1998.

\bibitem{DFZ}
P.~Di~Francesco and R.~Kedem.
\newblock Q-systems, heaps, paths and cluster positivity.
\newblock Available on the arXiv at {\tt arXiv:0811.3027}.

\bibitem{EKLP}
N.~Elkies, G.~Kuperberg, M.~Larsen, and J.~Propp.
\newblock Alternating-sign matrices and domino tilings. {I}.
\newblock {\em J. Algebraic Combin.}, 1(2):111--132, 1992.

\bibitem{FZ}
S.~Fomin and A.~Zelevinsky.
\newblock The {L}aurent phenomenon.
\newblock {\em Adv. in Appl. Math.}, 28(2):119--144, 2002.

\bibitem{Galeupdate}
D.~Gale.
\newblock Somos sequence update.
\newblock {\em Math. Intelligencer}, 13(4):49--50, 1991.

\bibitem{Gale}
D.~Gale.
\newblock The strange and surprising saga of the {S}omos sequences.
\newblock {\em Math. Intelligencer}, 13(1):40--42, 1991.

\bibitem{Hirota}
R.~Hirota.
\newblock {\em The direct method in soliton theory}.
\newblock Cambridge University Press, San Diego, CA, 2004.
\newblock Cambridge Tracts in Mathematics, No. 155.

\bibitem{Hone}
A.~N.~W. Hone.
\newblock {\em Discrete dynamics, integrability and integer sequences}.
\newblock Imperial College Press, in preparation.

\bibitem{KOS}
R.~Kenyon, A.~Okounkov, and S.~Sheffield.
\newblock Dimers and amoebae.
\newblock {\em Ann. of Math. (2)}, 163(3):1019--1056, 2006.
\newblock Available on the arXiv at {\tt math-ph/0311005}.

\bibitem{Kuo}
E.~H. Kuo.
\newblock Applications of graphical condensation for enumerating matchings and
  tilings.
\newblock {\em Theoret. Comput. Sci.}, 319(1-3):29--57, 2004.
\newblock Available on the arXiv at {\tt math.CO/0304090}.

\bibitem{Pr2}
J.~Propp.
\newblock The combinatorics of frieze patterns and {M}arkoff numbers.
\newblock Available on the arXiv at {\tt math.CO/0511633}.

\bibitem{Pr1}
J.~Propp.
\newblock A reciprocity theorem for domino tilings.
\newblock {\em Electron. J. Combin.}, 8(1):Research Paper 18, 9 pp.
  (electronic), 2001.
\newblock Available on the arXiv at {\tt math.CO/0104011}.

\bibitem{PW2}
J.~Propp and D.~Wilson.
\newblock Coupling from the past: a user's guide.
\newblock In {\em Microsurveys in discrete probability ({P}rinceton, {NJ},
  1997)}, volume~41 of {\em DIMACS Ser. Discrete Math. Theoret. Comput. Sci.},
  pages 181--192. Amer. Math. Soc., Providence, RI, 1998.

\bibitem{PW1}
J.~G. Propp and D.~B. Wilson.
\newblock Exact sampling with coupled {M}arkov chains and applications to
  statistical mechanics.
\newblock {\em Random Structures Algorithms}, 9(1-2):223--252, 1996.

\bibitem{RR}
D.~P. Robbins and H.~Rumsey, Jr.
\newblock Determinants and alternating sign matrices.
\newblock {\em Adv. in Math.}, 62(2):169--184, 1986.

\bibitem{SS}
N.~Saitoh and S.~Saito.
\newblock General solutions to the {B}\"acklund transformation of {H}irota's
  bilinear difference equation.
\newblock {\em J. Phys. Soc. Japan}, 56(5):1664--1674, 1987.

\bibitem{sloane}
N.~J.~A. Sloane and S.~Plouffe.
\newblock {\em The encyclopedia of integer sequences}.
\newblock Academic Press Inc., San Diego, CA, 1995.
\newblock http://www.research.att.com/$\sim$njas/sequences/index.html.

\bibitem{Sp}
D.~E. Speyer.
\newblock Perfect matchings and the octahedron recurrence.
\newblock {\em J. Algebraic Combin.}, 25(3):309--348, 2007.
\newblock Available on the arXiv at {\tt math.CO/0402452}.

\bibitem{Stanley}
R.~P. Stanley.
\newblock {\em Enumerative combinatorics. {V}ol. 1}, volume~49 of {\em
  Cambridge Studies in Advanced Mathematics}.
\newblock Cambridge University Press, Cambridge, 1997.

\bibitem{xin}
G.~Xin.
\newblock Proof of the {S}omos-4 {H}ankel determinants conjecture.
\newblock {\em Adv. in Appl. Math.}, 42(2):152--156, 2009.

\bibitem{Z}
A.~Zabrodin.
\newblock A survey of {H}irota's difference equations.
\newblock Available on the arXiv at {\tt solv-int/9704001}.

\end{thebibliography}

\end{document}